\documentclass[12pt]{amsart}

\usepackage{amssymb,amsfonts,eucal}
\usepackage{comment}
\usepackage{amsmath}
\usepackage{mathtools}
\mathtoolsset{showonlyrefs}
\usepackage[bookmarks]{hyperref}
\usepackage[usenames,dvipsnames,svgnames,table]{xcolor}
\usepackage{overpic}

\numberwithin{equation}{section}
\numberwithin{figure}{section}

\newcommand{\us}{\underline\sigma}
\newcommand{\os}{\overline\sigma}

\DeclareMathOperator{\tr}{tr}

\DeclareMathOperator{\diam}{diam}
\DeclareMathOperator{\vol}{vol }
\DeclareMathOperator{\cof}{cof}

\DeclareMathOperator{\rank}{rank}

\newcommand{\intomt}{\int_{\Omega_t}}

\newcommand{\omt}{\Omega_t}
\newcommand{\barominf}{\overline{\Omega}_\infty}

\newcommand{\rr}{\mathbb{R}}
\newcommand{\mm}{\mathbb{M}}
\newcommand{\cc}{\mathcal C}
\newcommand{\bb}{\mathcal B}
\newcommand{\oo}{\mathcal{O}}
\newcommand{\eps}{\varepsilon}

\newcommand{\xx}{{\mathcal X}}
\newcommand{\dd}{{\mathcal D}}

\newcommand{\glt}{\rm{GL}^+(3,\rr)}
\newcommand{\slt}{\rm{SL}(3,\rr)}

\newtheorem{theorem}{Theorem}
\newtheorem{lemma}{Lemma}
\newtheorem{corollary}{Corollary}

\theoremstyle{remark}
\newtheorem*{remark}{Remark}

\begin{document}
  
  \title[Affine motion of ideal fluids]{Global existence and asymptotic behavior\\ of 
  affine
  motion of 3D ideal fluids\\
   surrounded by vacuum}

 \author{Thomas C.\ Sideris}
\address{Department of Mathematics\\
University of California\\
Santa Barbara, CA 93106\\USA} 
\email{sideris@math.ucsb.edu}

\date{February 26, 2017}

\thanks{Final publication  available at Springer via http://dx.doi.org/10.1007/s00205-017-1106-3}


\begin{abstract}
The 3D compressible and incompressible Euler equations with a physical vacuum free boundary condition
and affine initial conditions reduce to a globally solvable Hamiltonian system of ordinary differential equations for the deformation gradient in $\glt$.  
The evolution of the fluid domain is described by a family
ellipsoids whose diameter grows at a rate proportional to time.  
Upon rescaling to a fixed diameter, the 
asymptotic limit of the  fluid ellipsoid is determined by  a  positive semi-definite quadratic form of rank $r=1$, 2, or 3,
corresponding to the asymptotic degeneration of the ellipsoid along $3-r$ of its principal axes.
 In the compressible case, the asymptotic limit has rank $r=3$, and asymptotic completeness holds,
when the adiabatic index $\gamma$ satisfies $4/3<\gamma<2$.
The number of possible degeneracies, $3-r$, increases with the value of the 
adiabatic index $\gamma$.  
In the incompressible case, affine motion reduces to geodesic flow in $\slt$ with the Euclidean metric.  
For incompressible affine swirling flow, there is a structural instability. 
Generically, when the  vorticity is nonzero, the domains degenerate along only one axis,
but the physical vacuum boundary condition  fails over a finite time interval.
 The rescaled fluid domains of irrotational motion can  collapse along two axes.


\end{abstract}

\maketitle

\section{Introduction}
\label{intro}

We shall consider the affine motion of ideal fluids in three spatial dimensions.  
An affine motion is a one-parameter family of deformations of the form 
\[
x(t,y)=A(t)y,
\]
 defined on the reference domain 
 \[
 B=\{y\in \rr^3:|y|<1\}.
 \]
The deformation gradient $A(t)$ takes values in $\glt$, the set of invertible $3\times3$ matrices with positive determinant.
We shall show that the equations of motion for ideal fluids surrounded by vacuum possess affine solutions
satisfying the physical vacuum boundary condition (i.e.\ positive normal acceleration of the free boundary of the fluid)
 in the case of compressible ideal {gases} and the case of incompressible fluids.
In both cases, the PDEs reduce to  globally solvable systems of    Hamiltonian ODEs for the deformation gradient, see equations \eqref{ODE} and \eqref{IODE}.

For affine motion, all hydrodynamical quantities are expressed explicitly  in terms of the deformation gradient $A(t)$ and the initial conditions, 
resulting in global finite energy classical solutions of
the initial free boundary value problem.  For incompressible motion, the evolution of $A(t)$ turns out to be geodesic flow in $\slt$, the set of  $3\times3$  matrices with  determinant equal to unity, with the Euclidean metric.
This should be seen as a special case of the general result of Arnold \cite{Arnold-1966} which provides an interpretation of the motion of {perfect} incompressible fluids
as geodesic flow in the space of volume-preserving deformations, see also Rouchon \cite{Rouchon-1991}.
There is  one proviso in the incompressible case, however.
The curvature of $A(t)$, as a curve in the vector space of $3\times 3$ matrices with the Euclidean metric, 
determines the sign of the pressure, with the potential effect that the physical vacuum boundary condition
may fail.  Thus, although solutions are global in time, there may, in general,  be a  time interval on which 
 the pressure vanishes or even becomes negative.
This will be illustrated for swirling flow  and shear flow in  Section \ref{swirlsec}.

For general affine motion, the fluid domains $\omt=A(t)B$ are ellipsoids.  We shall show that their diameters grow at a rate proportional to time,
provided the physical vacuum boundary condition holds.  
This improves upon the lower bound obtained by the author in  \cite{sideris-2014} for general flows.
The growth of the diameter, together with the form of the aforementioned ODEs,
suggests that $\lim_{t\to\infty}\|A''(t)\|=0$, and therefore, we are lead to consider the existence of asymptotic states of the
form 
\[
A_\infty(t)=A_0+tA_1
\]
 such that 
 \[
 \lim_{t\to\infty}\|A_\infty(t)\|=+\infty\quad\text{and}\quad \lim_{t\to\infty}\|A(t)-A_\infty(t)\|=0.
 \]   
Given this asymptotic behavior, we see that 0, 1, or 2 of the prinicipal axes of the rescaled ellipsoids $t^{-1}\omt$ may collapse,
depending upon the rank of the matrix $A_1$.  In the incompressible case, where the volume of $\omt$ is constant, at least one axes must collapse.
 After scaling down the domains to $t^{-1}\omt$, we shall see a full breakfast menu of asymptotic states:
 eggs, pancakes, and sausages.

In the compressible case, the set of possible asymptotic states increases with the adiabatic index $\gamma$.
We shall show that  any asymptotic state $A_\infty(t)$, with $\lim_{t\to\infty}\det A_\infty(t)=+\infty$,
is the limit of a  solution, for values of the adiabatic index $\gamma>5$.
If $\gamma>4/3$, then any state with $\det A_1>0$ is the limit of a unique solution.  Moreover, if $4/3<\gamma<2$, there
is a {\em wave operator}, i.e.\ a bijection between initial data $(A(0),A'(0))$ with $\det A(0)>0$ and asymptotic asymptotic states with $\det A_1>0$,
and in addition, asymptotic completeness holds.
If $\gamma>2 $ or $3$,  then there exist asymptotic states with $\rank A_1=2$ or $1$, respectively, which are approached by a unique solution.
These results are proven in Section \ref{cpatinfsec},
 using a fixed point argument applied to a suitable Cauchy problem at infinity involving the quantity $A(t)-A_\infty(t)$.

We do not attempt a classification of asymptotic behavior in the incompressible case.
Instead, we examine the case of  affine swirling flow, 
where a full range of asymptotic behavior can be found.
Generically, the rescaled fluid domains
collapse along one axis as $t\to\pm\infty$, but there is a bounded time interval within which the pressure becomes negative and the vacuum
boundary condition fails, see Theorem \ref {incswirl}.  For the sub-case of irrotational axi-symmetric flow, the pressure remains positive for all times,
but the rescaled fluid domains collapse along two axes as $t\to \text{$+\infty$ or $-\infty$}$ 
and along only one axis as $t\to \text{$-\infty$ or $+\infty$}$, depending upon the initial conditions,
 see Theorem \ref{incirrot}.  Thus, at least on the level
of ODEs, there is a structural instability in passing from irrotational to rotational affine swirling flow.
We also consider the special case of shear flow, where solutions are lines in $\slt$ and the pressure is identically zero.

The use of affine hydrodynamical motions in free space is a well-established technique for gaining a basic understanding of qualitative behavior, see
for example the expository article of Majda  \cite{Majda-1986} devoted to incompressible flow.
Of course in free space, these solutions have infinite energy.  In the present situation, affine finite energy solutions are  constructed in bounded moving domains
with prescribed boundary conditions.  This  restricts the evolution of the deformation gradient and the related hydrodynamical quantities.
Liu \cite{Liu-1996} considered the special case of spherically symmetric affine solutions, where the deformation gradient $A(t)$ is a multiple of the identity,
in connection with damped compressible flow surrounded by vacuum to recover Darcy's Law in the asymptotic limit.
The author used this same ansatz to provide a explicit example of spherical domain spreading for the vacuum free boundary problem in the (undamped)
compressible case
in \cite{sideris-2014}.  The present work demonstrates that a much richer spectrum of asymptotic behavior is possible
for  non-spherically symmetric affine motion.  In particular,
the are no nontrivial spherically symmetric affine solutions in the incompressible case.

Formation of singularities for {\em classical } large and small amplitude 3D compressible flow with a non-vanishing constant state outside a compact region
was established by the author in \cite{sideris-1985}.  Makino, Ukai, and Kawashima obtained an analogous result for a smooth compactly
supported disturbance moving into a vacuum state, see \cite{Makino-Ukai-Kawashima-1986}, \cite{Makino-Ukai-Kawashima-1987}.
A crucial role in singularity formation for these classical solutions is played by the constant
propagation speed of signals at the boundary of the disturbance  determined by the constant sound speed at infinity, which strongly suggests
the presence of compressive shocks at the front.
In 1D, this has been well-understood since  the work of Lax \cite{Lax-1964}, 
and  Christodoulou's pioneering work \cite{Christodoulou-2009} confirms this for irrotational relativistic fluids in 3D.

Serre \cite{Serre-1997}, \cite{Serre-2015}
and Grassin \cite{Grassin-1998} study the  motion of compressible fluids in $H^s(\rr^d)$, $s>1+d/2$,
with compactly supported density.  Restriction of these solutions to the support of of the density yields
a solution of the vacuum Euler equation.  But here again, the free boundary has zero acceleration.
They are able to construct global classical solutions by perturbing from  expansive background velocities
which satisfy the vectorial Burgers equation.  Serre \cite{Serre-2015} also uses this approach to
demonstrate finite time development of boundary singularities.

For the physical vacuum free boundary problem,  the free boundary has  nonzero acceleration, 
the solutions to the problem are merely weak in all of $\rr^3$, and   the above-mentioned results do not apply.
 It is unknown whether non-affine solutions to the physical vacuum free boundary problem develop singularities 
 within the fluid domain in finite time.
 Affine solutions, by virtue of their simplicity, do not allow for the development of small spatial structures,
 which rules out shock formation.  Nor, as we shall demonstrate, do 
 affine motions lead to finite time collapse or blow-up of the fluid domain.

The challenging problem of local well-posedness for the  physical vacuum
initial  free boundary  value problem for {ideal} fluid motion has been
exhaustively studied by a number of authors in recent years.  Wu considered the full water wave problem with gravity, in two 
and three dimensions, see  \cite{Wu-1997}, \cite{Wu-1999}, respectively.
Christodoulou and Lindblad initiated the study of the vacuum free boundary problem 
for incompressible flow without gravity in \cite{Christodoulou-Lindblad}.  Adopting a geometric point of view, they establish
key {\em a priori} estimates for Sobolev norms of solutions and  the second fundamental form of the free boundary of the fluid domain.  
Using this framework, Lindblad established local well-posedness for the linearized problem in  \cite{Lindblad-2003-2},
and he subsequently resolved local well-posedness for the nonlinear problem using Nash-Moser iteration in \cite{Lindblad-2005-2}.
Coutand and Skholler provided an alternative proof, with and without surface tension, which avoided the use of Nash-Moser, see 
\cite{Coutand-Shkoller-2007} and \cite{Coutand-Shkoller-2010}.  Local well-posedness for compressible liquids (nonzero fluid density
on the free boundary) was established by Lindblad, in the linearized case \cite{Lindblad-2003-1}
 and then in the full nonlinear case \cite{Lindblad-2005-1}.
 Finally, the case of ideal gases (vanishing fluid density on the free boundary), with an isentropic equation of state,
 was first studied by Coutand, Lindblad, and Shkoller, 
\cite{Coutand-Shkoller-Lindblad-2010}, who established {\em a priori} estimates.
Coutand and Shkoller then obtained local well-posedness for isentropic gases using parabolic regularization, 
\cite{Coutand-Shkoller-2012}.
Jang and Masmoudi solved the one-dimensional version of the problem in \cite{Jang-Masmoudi-2012},
and they later provided a proof in the multidimensional setting based on weighted energy estimates in \cite{Jang-Masmoudi-2015}.
All of these works rely heavily on the physical vacuum boundary condition:  {for incompressible  fluids the normal derivative of the pressure is negative on the
boundary, and for gases the normal derivative of the enthalpy must be negative}.

Recently, Had\v{z}i\'{c} and Jang \cite{Hadzic-Jang} have established the global-in-time nonlinear stability in Sobolev spaces
of the affine solutions to the compressible Euler equations constructed in Theorem \ref{caffinesol},
when the adiabatic index satisfies $1<\gamma\le5/3$.  These solutions satisfy the physical vacuum boundary condition.
The range of indices is determined by the asymptotic behavior of the underlying affine solution.

\section{Notation}
The set of all $3\times3$ matrices over $\rr$ will be denoted by $\mm^3$.  
Given $A\in\mm^3$, its determinant, trace, transpose, inverse (if it exists), and cofactor matrix will be written $\det A$, $\tr A$, $A^\top$, $A^{-1}$, and $\cof A$, respectively.  We define $A^{-\top}=(A^{-1})^\top$.
The vector space $\mm^3$ is isomorphic to $\rr^9$, and the Euclidean 
(or Hilbert-Schmidt) norm of an element $A\in\mm^3$ is $|A|=(\tr AA^\top)^{1/2}$.
The operator norm of $A\in\mm^3$ will be denoted by $\|A\|$.  These norms are equivalent.

We denote the identity component of the  general linear group by
\[
\glt=\{A\in\mm^3: \det A>0\},
\]
 and  the special linear group by 
 \[
 \slt=\{A\in\glt:\det A=1\}.
 \]

We adopt the standard notation $x\lesssim y$ to denote two nonnegative functions $x$ and $y$ for which there exists a generic constant $c>0$ such that $x\le c y$.
We write $x\sim y$ if $x\lesssim y$ and $y\lesssim x$.   The notation $\oo(y)$ will be used to denote a quantity $x$ with the property $x\lesssim y$.

\section{Global Existence of Affine Solutions}

\subsection{Kinematics}

We shall consider affine motions 
\begin{equation}
x(t,y)=A(t)y,\quad y\in B=\{y\in\rr^3:|y|=1\}.
\end{equation}
The as yet unknown deformation gradient $D_yx(t,y)=A(t)$ satisfies the minimal requirement
\begin{equation}
 A\in C(\rr,\glt)\cap C^2(\rr,\mm^3).
\end{equation}
Incompressible motion is volume preserving, and so in this case, we will require
\begin{equation}
 A\in C(\rr,\slt)\cap C^2(\rr,\mm^3).
\end{equation}
The  domain occupied by the fluid at time $t$ is the image $\omt$ of the reference domain $B$
under the deformation $x(t,\cdot)$, that is,
\begin{equation}
\omt=\{x(t,y)\in\rr^3:y\in B\}.
\end{equation}
Using the polar decomposition, we may write $A(t)=(A(t)A(t)^\top)^{1/2}R(t)$,
where $R(t)$ is a rotation, and so for affine motion, we have
\begin{equation}
\omt=A(t)B=(A(t)A(t)^\top )^{1/2}R(t)B=(A(t) A(t)^\top )^{1/2}B.
\end{equation}
This shows that $\omt$ is an ellipsoid whose principal axes are oriented in the directions of an
orthonormal  system of eigenvectors
for  the positive definite and symmetric  matrix $(A(t) A(t)^\top )^{1/2}$, called the stretch tensor.
In a coordinate frame determined by the eigenvectors, we have that
\begin{equation}
\omt=\left\{x\in\rr^3:\sum_{i=1}^3(x_i/\lambda_i(t))^2<1\right\},
\end{equation}
where $\{\lambda_i(t)\}_{i=1}^3$ are the  (strictly positive) eigenvalues of the stretch tensor.
Therefore, the diameter of $\omt$ is equal to twice the largest eigenvalue of the stretch tensor.
This implies that 
\begin{equation}
\label{diamsimtr}
\diam\omt\sim \tr(A(t) A(t)^\top )^{1/2}=\sum_{i=1}^3\lambda_i(t).
\end{equation}

In material coordinates, the velocity associated to this motion is 
\begin{equation}
\label{matvel} 
u(t,x(t,y))=\frac{d}{dt}x(t,y)=A'(t)y,\quad y\in B,
\end{equation}
or equivalently,
\begin{equation}
\label{m2v}
u(t,x)=A'(t)A(t)^{-1}x,\quad x\in \omt,
\end{equation}
in spatial coordinates.
It follows from \eqref{matvel} that the material time derivative of the velocity is
\begin{align}
\label{mattimeder}
D_tu(t,x)&=\left.\frac{d}{dt}u(t,x(t,y))\right|_{y=A(t)^{-1}x}\\
&=\left.\frac{d}{dt}A'(t)y\right|_{y=A(t)^{-1}x}\\
&=\left.A''(t)y\right|_{y=A(t)^{-1}x}\\
&=A''(t)A(t)^{-1}x.
\end{align}

From \eqref{m2v}  we  see that for affine motion,
 the velocity gradient $D_xu(t,x)$ is spatially homogeneous.  Define
 \begin{equation}
\label{velgraddef}
L(t)=D_xu(t,x)=A'(t)A(t)^{-1}.
\end{equation}
We note that by \eqref{velgraddef}
\begin{equation}
\label{defgradode}
A'(t)=L(t)A(t),
\end{equation}
and also
\begin{equation}
\label{divform}
\tr L(t)=\nabla\cdot u(t,x).
\end{equation}
Denote the Jacobian by
\begin{equation}
J(t)=\det A(t).
\end{equation}
Then $J(t)>0$, $t\in\rr$, since $A(t)\in\glt$.  It follows from \eqref{defgradode} that
\begin{equation}
\label{dettimederiv}
J'(t)=\tr L(t)J(t), \quad J(0)=\det A(0).
\end{equation}

\subsection{Rescaled Asymptotic Fluid Domains}

Ultimately, we shall construct affine motions $A(t)$ with the property that 
\begin{equation}
\label{asympstlim}
\lim_{t\to\infty}\|A(t)-A_\infty(t)\|=0,
\end{equation}
for some affine asymptotic state of the form 
\begin{equation}
A_\infty(t)=A_0+tA_1,\quad A_0,A_1\in\mm^3, \quad A_1\ne0.
\end{equation}

  If $\omt=A(t)B$ is the family of fluid domains, define the rescaled asymptotic fluid domain
\begin{equation}
\barominf=\lim_{t\to\infty}t^{-1}\omt=\{x\in\rr^3: x=\lim_{t\to\infty}t^{-1}A(t)y,\;\text{for some $y\in B$}\}.
\end{equation}
 If \eqref{asympstlim} holds, then $\barominf$ is simply the image of $B$ under $A_1$.
 Thus, in the appropriate coordinate frame, we will have
 \begin{equation}
\barominf=\{x\in\rr^3: \sum_{i=1}^r(x_i/\lambda_i)^2<1; \; x_i=0, \; r<i\le 3\},
\end{equation}
 where $\{\lambda_i\}_{i=1}^r$ are the nonzero eigenvalues of the positive semi-definite symmetric matrix $(A_1A_1^\top)^{1/2}$.
 We shall  have $A_1\ne0$,
provided the vacuum boundary condition holds. 
 Thus, the domain $\barominf$ will be an ellipsoid ($r=3$), an ellipse in some two-dimensional subspace ($r=2$), or a line segment
 ($r=1$) in some one-dimensional subspace.

\subsection{Compressible case}
The compressible Euler equations with an equation of state for an ideal gas  are
   
\begin{align}
\label{cpde1}
&\rho D_t u + \nabla p=0,\\
\label{cpde2}
&D_t\rho+\rho\nabla\cdot u=0,\\
\label{cpde3}
&D_t\eps+(\gamma-1)\eps\nabla\cdot u=0,\\
\label{stateeqn}
&p=p(\rho,\eps)=(\gamma-1)\rho\eps,\quad \gamma>1,
\end{align}
  
in which $u$, $\rho$, $\eps$, and $p$ are the fluid velocity vector field, mass density, specific internal energy, and pressure, respectively.
The operator $D_t$ stands for the material time derivative $D_t=\partial_t+u\cdot\nabla$.
The constant $\gamma$ is the adiabatic index.  
The vacuum free boundary problem consists in solving the system
\eqref{cpde1}, \eqref{cpde2}, \eqref{cpde3}, \eqref{stateeqn} in a regular open space-time region of the form
$\cc_T=\{(t,x)\in\rr\times\rr^3: x\in\omt,\;|t|<T\}$, where $\omt\subset\rr^3$ is a family of regular, simply connected,
 open sets with a well-defined unit outward  normal $\eta_x(t,x)$
for $x\in\partial\omt$.
The lateral (free) boundary $\bb_T$ of $\cc_T$ has a well-defined  normal $\eta=(\eta_t,\eta_x)\in\rr\times\rr^3$.  
The space-time velocity vector field $(1,u)$
is parallel to $\bb_T$:
\begin{equation}
\label{pbc}
\eta(t,x)\cdot(1,u(t,x))=0,\quad (t,x)\in\bb_T.
\end{equation}
The physical vacuum boundary condition is
\begin{align}
\label{vbc1}
&p(t,x)=0,&& (t,x)\in\bb_T,\\
\label{vbc2}
 &D_n\eps(t,x)=\eta_x(t,x)\cdot \nabla\eps(t,x)<0, &&(t,x)\in\bb_T.
\end{align}

We are going to construct global solutions of this system in the class of affine deformations.
The next three preparatory lemmas will simplify the statement of this result.
The first two concern the initial data for the density and internal energy, and the third
establishes the evolution of the deformation gradient.

\begin{lemma}\label{raddata}
Let 
\[
\mathcal Y=\{f\in C^0[0,1]\cap C^1[0,1): f(s)>0, \; 
s\in[0,1),\; f'(0)=f(1)=0\}.
\]
Choose a function $ \rho_0\in\mathcal Y$ such that
\begin{equation}
\label{brbb}
  0<\lim_{s\to 1^-}(1-s)^{-\delta} \rho_0(s)<\infty,\quad\text{for some}\quad \delta>0.
\end{equation}
If
\begin{equation}
\label{barepsdef}
  \eps_0(s)=\frac{\int_s^1\varsigma \rho_0(\varsigma)\;d\varsigma}{(\gamma-1) \rho_0(s)},
\end{equation}
then $ \eps_0\in\mathcal Y$, and
\begin{equation}
\label{epsbb}
 \eps_0'(1)
 =-[(\gamma-1)(1+\delta)]^{-1}<0.
\end{equation}
\end{lemma}
\begin{proof}
Since $ \rho_0\in\mathcal Y$, it is clear that $ \eps_0\in C^1[0,1)$ and   that $ \eps_0'(0)=0$.

Let $L=\lim_{s\to 1^-}(1-s)^{-\delta} \rho_0(s)$.  By l'H\^opital's rule 
\begin{equation}
\label{lhop}
\lim_{s\to 1^-}\frac{\int_s^1\varsigma \rho_0(\varsigma)\;d\varsigma}{(1-s)^{1+\delta}}
=\lim_{s\to 1^-}\frac{-s \rho_0(s)}{-(1+\delta)(1-s)^\delta}=\frac{L}{1+\delta}.
\end{equation}
Since by \eqref{barepsdef}
\begin{equation}
\label{barepsid}
\eps_0(s)= \frac{\int_s^1\varsigma \rho_0(\varsigma)\;d\varsigma}{(1-s)^{1+\delta}}
\frac{(1-s)^\delta}{(\gamma-1) \rho_0(s)} (1-s),
\end{equation}
it follows  from  \eqref{lhop} that $\lim_{s\to 1^-}\eps_0(s)=0$,
and so $ \eps_0\in\mathcal Y$.  Also from \eqref{barepsdef} and \eqref{barepsid}, we have
\begin{align}
\lim_{s\to 1^-}\frac{ \eps_0(s)- \eps_0(1)}{s-1}
&=-\lim_{s\to 1^-}\frac{\int_s^1\varsigma \rho_0(\varsigma)\;d\varsigma}{(1-s)^{1+\delta}}
\cdot \frac{(1-s)^\delta}{(\gamma-1) \rho_0(s)}\\
&=-\frac{L}{1+\delta}\cdot\frac1{(\gamma-1)L}\\
&=\frac{-1}{(\gamma-1)(1+\delta)}.
\end{align}
Thus, $ \eps_0'(1)$ exists, and \eqref{epsbb} holds.
  \end{proof}

\begin{lemma}
\label{ssdata}
Let $A\in\glt$ and set $\Omega=\{Ay\in\rr^3:y\in B\}$.  Define the  functions
\begin{equation}
\label{1d23d}
\rho(x)= \rho_0(|A^{-1}x|),\quad \eps(x)= \eps_0(|A^{-1}x|),\quad x\in\Omega,
\end{equation}
in which $ \rho_0, \eps_0\in\mathcal Y$ and $\eps_0$ is defined by \eqref{barepsdef}, (see Lemma \ref{raddata}).
Then the  functions $\rho$, $\eps$ are nonnegative, they belong to 
$C^0(\overline{\Omega})\cap C^1( \Omega )$, and they vanish on $\partial\Omega$.
The function $\eps$ satisfies the  condition
\begin{equation}
\label{epszbb}
D_n\eps(x)= \eps'_0(1)<0,\quad x\in\partial\Omega.
\end{equation}
\end{lemma} 

\begin{proof}
Since $\rho_0,\eps_0\in\mathcal Y$,  
the   functions $\rho(x)$, $\eps(x)$  belong to $C^0(\overline{\Omega})\cap C^1(\Omega)$.
 The unit outward normal $n(x)$ at a point $x\in\partial \Omega $
is $n(x)=|A^{-\top} A^{-1}x|^{-1}A^{-\top }A^{-1}x$.  From \eqref{1d23d}, we have $\nabla\eps_0(x)= \eps'_0(1)n(x)$,
for $x\in\partial\Omega$, and so $D_n\eps_0(x)=\nabla\eps_0(x)\cdot n(x)= \eps_0'(1)<0$, for $x\in\partial \Omega $.
  \end{proof}

\begin{lemma}
Let $\gamma>1$ be given.    For arbitrary initial data
\begin{equation}
(A(0),A'(0))\in\glt\times\mm^3,
\end{equation}
the system
\begin{equation}
\label{ODE}
A''(t)= (\det A(t))^{1-\gamma}A(t)^{-\top}
\end{equation}
has a unique global solution $A\in C(\rr,\glt)\cap C^\infty(\rr,\mm^3)$.  The solution satisfies the conservation law
\begin{equation}
\label{compenergy}
E(t)\equiv\frac12\tr A'(t)A'(t)^\top+{(\gamma-1)^{-1}\det A(t)^{1-\gamma}}
=E(0).
\end{equation}

\end{lemma}

\begin{proof}
Let $A\in\mm^3$, and set
 $C=\cof A$.  Fix  indices $(i,j)$.  Using the cofactor expansion across the $i^{\mbox{\tiny th}}$ row of $A$, we have
\begin{equation}
\label{cofexpan}
\det A=\sum_{\ell=1}^3A_{i\ell}C_{i\ell}.
\end{equation}
By definition, the cofactor $C_{i\ell}$ is independent of the $(i,j)^{\mbox{\tiny th}}$ entry $A_{ij}$, for $\ell=1,2,3$.  
Thus, regarding  $\det A$ as a function
from $\mm^3$ into $\rr$, we  have from \eqref{cofexpan}
that
\begin{equation}
\label{detderiv}
\frac{\partial}{\partial A_{ij}}\det A= C_{ij}.
\end{equation}
For $A\in\glt$,  the standard formula
\begin{equation}
\label{invform}
A^{-1}=(\det A)^{-1}(\cof A)^\top
\end{equation}
allows us to express the nonlinearity as 
\begin{equation}
\label{nonlinform}
N(A)=(\det A)^{1-\gamma}A^{-\top}
=(\det A)^{-\gamma}\cof A,
\end{equation}
from which it is clear that $N(A)$ is a $C^\infty$ function of $A$ on $\glt$.  

Writing the system \eqref{ODE} in first order form in the variables $(A_1,A_2)=(A,A')\in\glt\times\mm^3$, we have
\begin{equation}
A_1'(t)=A_2(t),\quad
A_2'(t)=N(A_1(t)).
\end{equation}
The vector field $F(A_1,A_2)=(A_2,N(A_1))$ maps the open set $\glt\times\mm^3\subset\mm^3\times\mm^3$ into $\mm^3\times\mm^3$,
and the preceding paragraph shows that this vector field is $C^\infty$ in $(A_1,A_2)$.
Therefore, the Picard existence and uniqueness theorem for ODEs
implies that the initial value problem for \eqref{ODE} has a unique local solution 
\begin{equation}
(A_1,A_2)\in C((-T,T),\glt\times\mm^3)\cap C^1((-T,T),\mm^3\times\mm^3), 
\end{equation}
for some $T>0$.  

Using \eqref{detderiv} and \eqref{nonlinform},
 we obtain
\begin{equation}
\label{nonlinpot}
N(A)=(\det A)^{-\gamma}\frac{\partial}{\partial A}\det A
=(1-\gamma)^{-1}\frac{\partial}{\partial A}(\det A)^{1-\gamma}.
\end{equation}
Combining \eqref{nonlinpot}   and \eqref{ODE}, we can now verify that the solution satisfies the conservation law \eqref{compenergy}:
\begin{align}
E'(t)&=\frac{d}{dt}\left[\frac12\sum_{i,j=1}^3A'_{ij}(t)^2+(\gamma-1)^{-1}(\det A(t))^{1-\gamma}\right]\\
&=\sum_{i,j=1}^3\left[A'_{ij}(t)A''_{ij}(t)+(\gamma-1)^{-1}\left.\frac{\partial}{\partial A_{ij}}(\det A)^{1-\gamma}\right|_{A=A(t)}A'_{ij}(t)\right]\\
&=\sum_{i,j=1}^3A'_{ij}(t)[A''_{ij}(t)-N(A(t))_{ij}]\\
&=0.
\end{align}
Since the energy satisfies $E(t)=E(0)>0$, for $t\in (-T,T)$, we see that $\tr A'(t) A'(t)^\top$ is uniformly bounded
above and that $\det A(t)$ is uniformly bounded below.  The boundedness of $\tr A'(t) A'(t)^\top$ implies that $\tr A(t) A(t)^\top 
\lesssim 1+t^2$.  This shows that $(A(t),A'(t))$ remains in a compact subset of the domain of the vector field $F$ over every bounded time interval.
It follows that $A$ can be extended to a unique global solution in the desired space.  
Finally, the smoothness of the nonlinearity implies that $A\in C^\infty(\rr,\mm^3)$.
  \end{proof}

\begin{remark}
The system \eqref{ODE} is time reversible.  If $A(t)$ is a solution with initial data $(A(0),A'(0))$, then $\tilde A(t)=A(-t)$ is a 
solution with initial data $(A(0),-A'(0))$.  This means that any statement which holds for all solutions as $t\to\infty$ will also hold
for all solutions as $t\to-\infty$.
\end{remark}

\begin{theorem}
\label{caffinesol}
Fix $\gamma>1$.  Given initial data 
\begin{equation}
(A(0),A'(0))\in\glt\times\mm^3,
\end{equation}
 let $A\in C(\rr,\glt)\cap C^\infty(\rr,\mm^3)$
be the global solution of \eqref{ODE}.  Define $\omt=A(t)B$ and $\cc=\{(t,x): t\in\rr,\;x\in\omt\}$.  Let $ \rho_0, \eps_0\in\mathcal Y$ 
with $ \eps_0$ defined by \eqref{barepsdef}.  Then the triple
\begin{align}
\label{udef}
&u(t,x)=A'(t)A(t)^{-1}x\\
\label{rhodef}
&\rho(t,x)= \rho_0(|A(t)^{-1}x|)/(\det A(t))\\
\label{epsdef}
&\eps(t,x))= \eps_0(|A(t)^{-1}x|)/(\det A(t))^{\gamma-1}
\end{align}
lies in  $C^0(\overline\cc)\cap C^1(\cc)$,  solves the compressible Euler equations
\eqref{cpde1}, \eqref{cpde2}, \eqref{cpde3}, \eqref{stateeqn} 
in $\cc$, and satisfies  the  boundary conditions \eqref{pbc}, \eqref{vbc1}, \eqref{vbc2}.
\end{theorem}

\begin{proof}
Since $A(t)\in\glt$ for $t\in\rr$,
Lemma \ref{ssdata} shows that $\rho$, $\eps$ are nonnegative functions on $\cc$ lying in  $C^0(\overline\cc)\cap C^1(\cc)$.
The  boundary condition \eqref{vbc1}  holds by the definition \eqref{stateeqn}, and 
$\eps$ satisfies \eqref{vbc2} by \eqref{epszbb}.

The form  \eqref{udef} for the velocity $u$ follows from \eqref{velgraddef}.
The velocity is $C^\infty$, and the boundary condition \eqref{pbc} holds since the domains $\omt$ are
obtained as the image of $B$ under the motion $x(t,y)=A(t)y$ determined by $u(t,x)$, see \eqref{m2v}.

It remains to verify the PDEs \eqref{cpde1}, \eqref{cpde2}, \eqref{cpde3}.  For this it is convenient to use material coordinates $(t,y)$
and to set $J(t)=\det A(t)$.
By  \eqref{rhodef}, \eqref{divform}, \eqref{dettimederiv}, we have
\begin{align}
D_t\rho(t,x)&=\left.\frac{d}{dt}\rho(t,A(t)y)\right|_{y=A(t)^{-1}x}\\
&=\left.\frac{d}{dt}J(t)^{-1} \rho_0(|y|)\right|_{y=A(t)^{-1}x}\\
&=-\tr L(t)J(t)^{-1} \rho_0(|A(t)^{-1}x|)\\
&=-\rho\nabla\cdot u(t,x).
\end{align}
This verifies \eqref{cpde2}.  An identical calculation yields \eqref{cpde3}.

We now turn to \eqref{cpde1}.  Note that we shall regard $u$ and $\nabla $ as column vectors in this calculation.
In \eqref{mattimeder}, we derived
\begin{equation}
\label{mattimeder2}
D_tu(t,x)=A''(t)A(t)^{-1}x.
\end{equation}
By \eqref{stateeqn}, \eqref{rhodef}, \eqref{epsdef}, we have
\begin{equation}
\label{pressform}
p(t,x)=J(t)^{-\gamma}\;  p_0(s(x)),\quad   p_0=(\gamma-1)  \rho_0 \eps_0,\quad s(x)=|A(t)^{-1}x|.
\end{equation}
From the definition \eqref{barepsdef}, it follows that
\begin{equation}
\label{barpressderiv}
  p_0'(s)=-s \rho_0(s).
\end{equation}
Combining \eqref{pressform}, \eqref{barpressderiv}, and \eqref{rhodef}, we compute the pressure gradient
\begin{align}
\label{pressgrad}
\nabla p(t,x)&=J(t)^{-\gamma}\;\nabla  [ p_0(s(x))]\\
&=J(t)^{-\gamma}\;  p_0'(s(x))\;\nabla s(x)\\
&=J(t)^{-\gamma}\;[-s(x)\; \rho_0(s(x))]\;\nabla s(x)\\
&=-J(t)^{-\gamma}\; \rho_0(s(x))\;(1/2)\nabla (s(x)^2)\\
&= -J(t)^{1-\gamma}\;\rho(t,x)\;A(t)^{-\top}A(t)^{-1}x.
\end{align}

Since $A(t)$ satisfies \eqref{ODE},  the formulas \eqref{mattimeder2} and \eqref{pressgrad} imply that
\begin{multline}
\rho(t,x)D_tu(t,x)+\nabla p(t,x)\\
= \rho(t,x)(A''(t)-J(t)^{1-\gamma}A(t)^{-\top})A(t)^{-1}x=0,
\end{multline}
and so
\eqref{cpde1} holds.
  \end{proof}

\begin{remark}
We note that \eqref{barpressderiv} is the key condition  behind this verification.
\end{remark}

\begin{remark}
We point out the role implicitly played by the vacuum boundary condition \eqref{vbc2}.
Consider \eqref{ODE} with the ``wrong sign" on the right-hand side.  
Then the energy density $E$ in \eqref{compenergy} would no longer be positive definite, and we
can lose the existence of global solutions.  Indeed, blow-up can occur in the spherically symmetric  case, $A(t)=\alpha(t)I$,
for a scalar $\alpha(t)$.
The preceding verification
  leads to a local solution of the PDEs  with negative internal energy and  pressure, violating the condition \eqref{vbc2}.
 
\end{remark}

\begin{remark}
In the  isentropic case, $p=\rho^{\gamma}$, i.e.\ $\eps =(\gamma-1)^{-1}\rho^{\gamma-1}$, the relation \eqref{barpressderiv} 
leads to an ODE for $\rho_0$ whose solution is 
\begin{equation}
\rho_0(s)=\left[\frac{\gamma-1}{2\gamma}(1-s^2)\right]^{1/(\gamma-1)}.
\end{equation}
Thus,  the  parameter in \eqref{brbb} is $\delta=1/(\gamma-1)$.  We also have 
\begin{equation}
\eps_0(s)=\frac1{2\gamma}(1-s^2).
\end{equation}
\end{remark}

\subsection{Incompressible case}
The  Euler equations for an incompressible {perfect} fluid take the form
\begin{align}
\label{ipde1}
&D_tu+\nabla p=0,\\
\label{ipde2}
&\nabla\cdot u=0.
\end{align}
These are again to be solved in a space-time cylinder $\cc_T$, as in the compressible case,
with the boundary conditions
\begin{align}
\label{ipbc}
&\eta(t,x)\cdot(1,u(t,x))=0, && (t,x)\in\bb_T,\\
\label{ivbc1}
&p(t,x)=0,&&(t,x)\in\bb_T,\\
\label{ivbc2}
&D_np(t,x)<0,&&(t,x)\in\bb_T.
\end{align}

The next lemma describes the evolution of the deformation gradient of affine solutions in the incompressible case.
 
\begin{lemma}
Given initial data $(A(0),A'(0))\in\slt\times\mm^3$ with  
\begin{equation}
\tr A'(0)A(0)^{-1}=0,
\end{equation}
the system
\begin{equation}
\label{IODE}
A''(t)=\Lambda(A(t))\;A(t)^{-\top},\quad  \Lambda(A(t))\equiv\frac{\tr(A'(t)A(t)^{-1})^2}{\tr (A(t)^{-\top}A(t)^{-1})},
\end{equation}
has a unique global solution $A\in C(\rr,\slt)\cap C^\infty(\rr,\mm^3)$.
The solution satisfies the conservation law
\begin{equation}
\label{ienergyid}
E_K(t)\equiv\frac12\tr A'(t)A'(t)^\top=E_K(0).
\end{equation}
If $E_K(0)>0$, the solution is a geodesic curve in $\slt$.  As a curve in $\mm^3$, its { curvature} is
\begin{equation}
\label{curvaturedef}
\kappa(t)=\frac{\tr (A'(t)A(t)^{-1})^2}{2E_K(0)\;(\tr A(t)^{-\top}A(t)^{-1})^{1/2}}.
\end{equation}
\end{lemma}
\begin{proof}
By \eqref{invform}, it follows that the right-hand side of \eqref{IODE} is a $C^\infty$ function of $A$
on $\glt$.  For the moment, take initial data in $\glt\times\mm^3$.  Arguing as in Theorem \ref{caffinesol}, we can construct a unique local solution of \eqref{IODE}  
$A\in C((-T,T),\glt)\cap C^2((-T,T),\mm^3)$, for some $T>0$.

We now show that if the initial data satisfies  $A(0)\in\slt$ and 
\[
\tr A'(0)A(0)^{-1}=0,
\]
 then $A\in C((-T,T),\slt)$.
Define 
\begin{equation}
L(t)=A'(t)A(t)^{-1},
\end{equation}
 for $t\in(-T,T)$.  
By \eqref{defgradode}, we have
\begin{equation}
A''(t)=L'(t)A(t)+L(t)A'(t),
\end{equation}
and so, since $A''(t)=\Lambda(A(t))A(t)^{-\top}$ by \eqref{IODE}, we get
\begin{align}
\label{ellode}
L'(t)&=[A''(t)-L(t)A'(t)]A(t)^{-1}\\
&=\Lambda(A(t))\;A(t)^{-\top}A(t)^{-1}-L(t)^2.
\end{align}
 This implies that 
 \begin{equation}
 \label{trlambpr}
\tr L'(t)=\Lambda(A(t)) \tr A(t)^{-\top}A(t)^{-1}-\tr L(t)^2=0,
\end{equation}
 by definition of $\Lambda(A(t))$.  By assumption on the initial data, we have $\tr L(0)=\tr A'(0)A(0)^{-1}=0$,
 and thus, 
 \begin{equation}
 \label{divvelvan}
\tr L(t)=0,\quad  t\in(-T,T).
\end{equation}
   By \eqref{dettimederiv}, \eqref{divvelvan}, we  see that $J(t)=\det A(t)$ satisfies $J'(t)=0$,
 and so $J(t)=J(0)=1$, since $A(0)\in\slt$.  We have proven that the solution satisfies
 $A\in C((-T,T),\slt)$.

Next, we verify the conservation law.  
By \eqref{detderiv} and \eqref{invform}, we have
\begin{equation}
\label{detderiv2}
A^{-\top}= (\det A)^{-1}\frac{\partial}{\partial A}\det A ,\quad A\in\glt.
\end{equation}
Thus, using  \eqref{IODE}  and \eqref{detderiv2}, we have
 \begin{align}
E'_K(t)&=\frac{d}{dt}\frac12\tr A'(t)A'(t)^\top\\
&=\sum_{i,j=1}^3A'_{ij}(t)A''_{ij}(t)\\
&=\Lambda(A(t))\sum_{i,j=1}^3(A(t)^{-\top})_{ij}A'(t)_{ij}\\
&=\Lambda(A(t))\sum_{i,j=1}^3(\det A)^{-1}\left.\frac{\partial}{\partial A_{ij}}\det A\right|_{A=A (t)}A'(t)_{ij}\\
&=\Lambda(A(t))\;(\det A(t))^{-1}\frac{d}{dt}\det A(t)\\
&=0,
\end{align}
since $\det A(t)=1$.  This proves \eqref{ienergyid}.

Now that \eqref{ienergyid} holds, we have that $\tr A(t)A(t)^\top\lesssim 1+t^2$.  Let $\lambda>0$ be an eigenvalue of the
positive definite symmetric matrix $A(t)A(t)^\top$.  The eigenvalues of $A(t)^{-\top}A(t)^{-1}=(A(t)A(t)^\top)^{-1}$ are the inverses
of the eigenvalues of $A(t)A(t)^\top$.  Thus, we have
\begin{equation}
\tr A(t)^{-\top}A(t)^{-1} \ge 1/\lambda \ge ( \tr A(t)A(t)^\top)^{-1}\gtrsim (1+t^2)^{-1}.
\end{equation}
It follows that $(A(t),A'(t))$ remains in a compact subset of the domain of the nonlinearity on every bounded time interval.
Therefore, the solution is global.  It lies in $C^\infty(\rr,\mm^3)$ thanks to the smoothness of the nonlinearity.

Since $\slt=\{A\in\mm^3:\det A=1\}$ is the level set of a smooth function, we see that $\slt$ is an embedded submanifold
of $\mm^3\approx\rr^9$.  Equation \eqref{detderiv2} says that $A^{-\top}$ is normal to $\slt$ at a point $A\in\slt$.
  Therefore, 
\begin{equation}
n(A)=(\tr A^{-\top}A^{-1})^{-1/2}A^{-\top}
\end{equation}
is a unit normal along $\slt$.  Equation \eqref{IODE} implies that the tangential component of
the acceleration vector $A''(t)$  vanishes.  In other words,  $A(t)$ is a geodesic.

We finish the proof with the verification of \eqref{curvaturedef}.
By \eqref{ienergyid}, the tangent vector $A'(t)$ has constant length 
\begin{equation}
(\tr A'(t)A'(t)^\top)^{1/2}=(2E_K(0))^{1/2}\equiv\tau>0.
\end{equation}
The reparameterized curve $\tilde A(s)=A(s/\tau)$ in $\mm^3$ has a unit length tangent, so its
curvature is given by the length of its acceleration vector $\tilde A''(s)$, that is
\begin{equation}
\kappa(s)=(\tr \tilde A''(s)\tilde A''(s)^\top)^{1/2}.
\end{equation}
The claim \eqref{curvaturedef} follows from this and \eqref{IODE}.
  \end{proof}

\begin{theorem}
\label{iaffinesol}
Given initial data $(A(0),A'(0))\in\slt\times\mm^3$ with 
\begin{equation}
\tr A'(0)A(0)^{-1}=0,
\end{equation}
let $A\in C(\rr,\slt)\cap C^\infty(\rr,\mm^3)$ be the global solution of the initial value problem for \eqref{IODE}.
Define $\omt=A(t)B$ and $\cc=\{(t,x):x\in\omt,\;t\in\rr\}$.  Then the pair
\begin{align}
&u(t,x)=A'(t)A(t)^{-1}x\\
&p(t,x)=\frac12\; \frac{\tr (A'(t)A(t)^{-1})^2}{\tr A(t)^{-\top}A(t)^{-1}}\;[1-|A(t)^{-1}x|^2]
\end{align}
solves the incompressible Euler equations \eqref{ipde1}, \eqref{ipde2} in $\cc$ and the boundary conditions
\eqref{ipbc}, \eqref{ivbc1}.  If the curvature defined in \eqref{curvaturedef}
is positive, then the boundary condition \eqref{ivbc2} also holds.  
\end{theorem}

\begin{proof}
This is a straightforward calculation.  As in the proof of Theorem \ref{caffinesol}, we have that
$u\in C^\infty$, the boundary condition \eqref{ipbc} holds, and $D_tu(t,x)=A''(t)A(t)^{-1}x$.

Since $\omt=A(t)B$, the boundary condition \eqref{ivbc1} is satisfied.  By definition,  the pressure is $C^\infty$, and its gradient is
\begin{equation}
\label{gradpressform}
\nabla p(t)=-\Lambda(A(t))\;A(t)^{-\top}A(t)^{-1}x,\; \mbox{with} \; \Lambda(A(t))= \frac{\tr (A'(t)A(t)^{-1})^2}{\tr A(t)^{-\top}A(t)^{-1}}.
\end{equation}
Therefore, since $A(t)$ is a solution of  \eqref{IODE}, we have that
\begin{equation}
D_tu(t,x)+\nabla p(t,x)=(A''(t)-\Lambda(A(t))A(t)^{-\top})A(t)^{-1}x=0,
\end{equation}
so that the PDE \eqref{ipde1} is satisfied.

From \eqref{velgraddef}, \eqref{divvelvan}, we obtain
\begin{equation}
\nabla\cdot u(t,x)=\tr A'(t)A(t)^{-1}=0,
\end{equation}
which verifies \eqref{ipde2}.

Since the unit normal along $\partial\omt$ is 
\begin{equation}
n(t,x)=A(t)^{-\top}A(t)^{-1}x/|A(t)^{-\top}A(t)^{-1}x|,
\end{equation}
the expression \eqref{gradpressform} yields
\begin{equation}
D_np(t,x)=-\Lambda(A(t)) |A(t)^{-\top}A(t)^{-1}x|^{1/2}, \quad x\in\partial\omt.
\end{equation}
Since $\Lambda(A(t))$ and $\kappa(t)$ share the same sign, we see that \eqref{ivbc2} holds when the curvature is positive.
  \end{proof}

\begin{remark}
The physical vacuum boundary condition \eqref{ivbc2} is not required for the global solvability of the  system of ODEs \eqref{IODE}  nor 
for the PDEs \eqref{ipde1}, \eqref{ipde2}.  
\end{remark}

\begin{remark}
Write $L(t)=D(t)+W(t)$ with $D(t)=\frac12(L(t)+L(t)^\top)$ and $W(t)=\frac12(L(t)-L(t)^\top)$.  The symmetric part $D(t)$ is called the strain rate tensor,
and the antisymmetric part $W(t)$ defines the vorticity vector $\omega(t,x)=\omega(t)$ through the operation
$W(t)=\frac12\omega(t)\times$.  Notice that
\begin{equation}
\tr L(t)^2= \tr D(t)^2+ \tr W(t)^2=\tr D(t)D(t)^\top-\tr W(t)W(t)^\top.
\end{equation}
Thus, it is apparent from \eqref{curvaturedef} that negative curvature and pressure can arise only if vorticity is present.  We shall see in Theorem \ref{incswirl}
that negative curvature is indeed possible.

Incompressible irrotational flows ($\omega=0$)  exist (at least locally) for the general vacuum free boundary problem, 
and  the pressure
remains positive within the fluid domain, by the maximum principle.  For irrotational affine motion, it is clear from the explicit formulas that 
the pressure and curvature are positive.  This will be further highlighted in
Theorem \ref{incirrot}.
\end{remark}

\begin{remark}
The solutions given in Theorems \ref{caffinesol} and \ref{iaffinesol} can be extended to  global weak solutions on $\rr\times\rr^3$ by setting all quantities to zero 
on the complement of the space-time fluid domain $\cc$.
\end{remark}

\section{Spreading of Fluid Domains}

In this section, we prove that for affine motion the diameters of the fluid domains $\omt$ grow at a rate proportional to time,
provided the vacuum boundary condition holds.
This improves upon the results for general flows previously given by the author  in \cite{sideris-2014} where only lower bounds were 
obtained.  We also obtain growth estimates for the volume of $\omt$ in the compressible case, by showing that the potential energy
decays to zero.

\begin{theorem}
\label{compspread}
If $A\in C(\rr,\glt)\cap C^2(\rr,\mm^2)$ is a solution of \eqref{ODE} and $\omt=A(t)B$, then
\begin{equation}
\label{tracegrowth}
\diam\omt\sim(\tr A(t)A(t)^\top)^{1/2} \sim 1+|t|,\quad t\in\rr,
\end{equation}
and
\begin{equation}
\label{detgrowth}
 1+|t|^p\lesssim \vol\omt\sim\det A(t) \lesssim 1+|t|^3,\quad t\in\rr,
\end{equation}
with
 \begin{equation}
 p=
 \begin{cases}
3, &\mbox{if $\;1< \gamma\le5/3$},\\
2/(\gamma-1),&\mbox{if $\;\gamma> 5/3$}.
\end{cases}
\end{equation}
\end{theorem}

\begin{proof}

Define the quantities
\begin{align}
&X(t)=\frac12\tr A(t) A(t)^\top ,\\
&E_K(t)=\frac12\tr A'(t)A'(t)^\top ,\\
&E_P(t)=(\gamma-1)^{-1}(\det A(t))^{1-\gamma}.
\end{align}
The identity \eqref{compenergy} can be written as
\begin{equation}
\label{compenergy2}
E(t)=E_K(t)+E_P(t)=E(0).
\end{equation}
Note that $X(t)>0$ and $E_P(t)>0$, for $t\in\rr$.  From \eqref{compenergy2}, we also have $E_K(t)=E(0)-E_P(t)<E(0)$, for $t\in\rr$.
It follows by direct calculation using  \eqref{ODE} that
\begin{equation}
\label{virialid}
X''(t)=2E_K(t)+3(\gamma-1)E_P(t).
\end{equation}
Using \eqref{compenergy2} and \eqref{virialid}, we can write
\begin{gather}
\label{convex}
X''(t)=2E(0)\left[\Theta(t)+3(\gamma-1)/2\;(1-\Theta(t))\right],\\
\intertext{with}
\label{convexparam}
 \Theta(t)=E_K(t)/E(0)\in[0,1).
\end{gather}
Thus, \eqref{convex} and \eqref{convexparam} imply that
\begin{equation}
\label{bracket}
X''(t)\in 2E(0)[\us,\os],
\end{equation}
in which
\begin{equation}
\us=\min(1,{3(\gamma-1)}/{2}),\quad \os=\max (1,{3(\gamma-1)}/{2}).
\end{equation}
It follows by integration of \eqref{bracket} that
\begin{gather}
\label{bracket2}
X'(t)-X'(0)\in 2E(0)\;t\;[\us,\os],\\
\label{bracket3}
X(t)-X'(0)t-X(0)\in E(0)\;t^2\;[\us,\os].
\end{gather}
By  \eqref{bracket2}, there exists a  $T>0$ such that
\begin{equation}
\label{expasympbehav}
 X'(t)\gtrsim t>0,\quad t\ge T.
\end{equation}
Since $X(t)>0$, \eqref{bracket3} implies that 
\begin{equation}
\label{exasympbehav}
X(t)\sim 1+t^2,\quad t\in\rr.
\end{equation}
  
This proves \eqref{tracegrowth}.

Of course, \eqref{tracegrowth} implies that the eigenvalues of the positive definite matrix $A(t)A(t)^\top$
are $\lesssim 1+ t^2$.  Thus, $\det A(t)= (\det A(t)A(t)^\top)^{1/2}\lesssim 1+|t|^3$, which proves the upper bound
in \eqref{detgrowth}.

By the Cauchy-Schwarz inequality, we have
\begin{align}
\label{csbound}
|X'(t)|&=|\tr A(t)A'(t)^\top|\\
&\le  (\tr A(t)A(t)^\top)^{1/2}(\tr A'(t)A'(t)^\top)^{1/2}\\
&=2 X(t)^{1/2}E_K(t)^{1/2},
\end{align}
and so by \eqref{convexparam}
\begin{equation}
\label{uudef}
U(t)\equiv\frac{X'(t)^2}{4E(0)X(t)}\le\Theta(t)<1.
\end{equation}
Cycling this additional restriction on the range of $\Theta(t)$ into \eqref{convex}, we obtain the improvement
\begin{equation}
X''(t)\ge 2E(0)\;[U(t)+\us(1-U(t))],
\end{equation}
 or equivalently,
 \begin{equation}
 \label{bootstrap}
\frac{X''(t)}{2E(0)}-U(t)\ge\us(1-U(t)).
\end{equation}

Differentiation of  the function $U(t)$ defined in \eqref{uudef} yields

\begin{equation}
U'(t)
\label{uode}
=\frac{X'(t)}{X(t)}\left\{\frac{X''(t)}{2E(0)}-U(t)\right\}.
\end{equation}
Combining \eqref{expasympbehav}, \eqref{bootstrap}, \eqref{uode}, we get 
\begin{equation}
U'(t)\ge\us\frac{X'(t)}{X(t)}[1-U(t)], \quad t\ge T.
\end{equation}

Integration of this differential inequality yields
\begin{equation}
[1-U(T)]\left[\frac{X(T)}{X(t)}\right]^{\us}\ge 1-U(t)\ge 1-\Theta(t)={E_P(t)}/{E(0)},\quad t\ge T.
\end{equation}
The lower bound of \eqref{detgrowth}, for positive times, is a consequence of this estimate, \eqref{exasympbehav}, and the definition of $E_P(t)$.
The estimate for negative times follows by time reversibility.
  \end{proof}
\begin{remark}
The identity \eqref{virialid} satisfied by the function $X(t)$ is the affine version of the integral identity used in \cite{sideris-1985}, \cite{sideris-2014},
insofar as
\begin{equation}
c_0\; X(t)= \frac12\intomt|x|^2\rho (t,x)dx, \quad c_0= \frac{4\pi}{3}\int_0^1s^5\rho_0(s)ds.
\end{equation}
In fact, the lemma holds for any globally defined general flow.
\end{remark}

\begin{remark}
It follows from Theorem \ref{compspread} and \eqref{stateeqn}, \eqref{rhodef}, \eqref{epsdef}, that the pressure satisfies
\begin{equation}
\|p(t,\cdot)\|_{L^\infty}\lesssim (1+|t|)^{-\gamma p}. 
\end{equation}
\end{remark}

\begin{remark}
Bounds for the potential energy in compressible flow were also investigated by Chemin in \cite{Chemin}.
\end{remark}

\begin{theorem}
\label{incompspread}
If $A\in C(\rr,\slt)\cap C^2(\rr,\mm^2)$ is a solution of \eqref{IODE} with $\kappa(t)\ge0$ for $t\ge T$ and $\omt=A(t)B$, then
\begin{equation}
\label{inctracegrowth}
\diam \omt\sim(\tr A(t)A(t)^\top)^{1/2} \sim 1+|t|,\quad t\ge T.
\end{equation}

\end{theorem}

\begin{proof}
Consider once again the function $X(t)=\frac12\tr A(t)A(t)^\top$.
As in the proof of Theorem \ref{compspread}, we obtain from \eqref{IODE},
\begin{equation}
X''(t)=2E_K(t)+3\Lambda(A(t)),\quad \Lambda(A(t))=\frac{\tr (A'(t)A(t)^{-1})^2}{\tr A(t)^{-\top}A(t)^{-1}}.
\end{equation}
Recall that $\kappa(t)\ge0$ implies that $\Lambda(A(t))\ge0$, and so we have that
\begin{equation}
X''(t)\ge 2E_K(t)=2E_K(0), \quad t\ge T.
\end{equation}

On the other hand, using $|\cdot|$ for the Euclidean norm, we have
$\tr A(t)^{-\top}A(t)^{-1}=|A(t)^{-1}|^2$, and so
\[
|\Lambda(A(t))|\lesssim |A'(t)|^2=2E_K(t)=2E_K(0).
\]

Therefore, we obtain
\begin{equation}
X''(t)\sim E_K(0),
\end{equation}
and then $X(t)\sim 1+t^2$, for $t\ge T$.
  \end{proof}

\section{Cauchy Problem at Infinity for Compressible Affine Motion}
\label{cpatinfsec}

We now consider  the asymptotic  behavior of solutions to the system \eqref{ODE}.
Having just shown that solutions $A(t)$ satisfy $\lim_{t\to\infty}\det A(t)=\infty$, it is reasonable to guess from \eqref{ODE} that
$\lim_{t\to\infty} A''(t)=0$.  This would suggest that the solution $A(t)$ approaches a free state of the form $A_\infty(t)=A_0+tA_1$, as $t\to\infty$.
In order to establish a result of this type, it is important to first understand the behavior of $N(A_\infty(t))$, as $t\to\infty$,
where $N(A)$ is the nonlinearity \eqref{nonlinform}.

\begin{lemma}
\label{asymptoticstates}
Let $A_0,A_1\in\mm^3$, and define $A_\infty(t)=A_0+tA_1$.  
Assume that 
\begin{equation}
\label{detgrow}
\lim_{t\to\infty}\det A_\infty(t)=+\infty.
\end{equation} 
Then
\begin{equation}
d\equiv\deg\det A_\infty(t)\in\{1,2,3\},
\end{equation}
$A_\infty(t)\in\glt$  for $t\gg1$, and
\begin{equation}
\|A_\infty(t)^{-1}\|\sim t^a,\quad t\gg1,\quad  \text{with}\quad a
=
\begin{cases}
-1, & \mbox{if $d=3$}\\
0,& \mbox{if $d=2$}\\
0\text{ or }1,&\mbox{if $d=1$}.
\end{cases}
\end{equation}

\end{lemma}

\begin{proof}
The assumption \eqref{detgrow} implies that $A_\infty(t)\in\glt$  for $t\gg1$.

Since $ A_\infty(t)$ is linear in $t$, we can write
\begin{equation}
\det A_\infty(t)=\sum_{j=0}^3\beta_jt^j.
\end{equation}
If $d=\deg A_\infty(t)$, then   \eqref{detgrow} implies that $d\in\{1,2,3\}$.
Note that the coefficient of $t^3$ is $\beta_3=\det A_1$.

Again since $ A_\infty(t)$ is linear in $t$, its cofactor matrix has the form:
\begin{equation}
\cof A_\infty(t)=\sum_{j=0}^2C_jt^j,
\end{equation}
and $\|\cof A_\infty(t)\|\sim t^b$, for $t\gg1$, $b\in\{0,1,2\}$.
Since

\begin{equation}
\label{cofeqn}
(\cof A_\infty(t))^\top A_\infty(t)=\det A_\infty(t)\;I,
\end{equation}
we have that $a=b-d$.
We can identify powers of $t$ in \eqref{cofeqn} to arrive at the system
\begin{align}
\label{coef3}
&C_2^\top A_1=\beta_3I=\det A_1\;I\\
\label{coef2}
&C_2^\top A_0+C_1^\top A_1=\beta_2I\\
\label{coef1}
&C_1^\top A_0+C_0^\top A_1=\beta_1I
\end{align}

  Notice that $d=3$ if and only if $A_1$ is invertible.  In this case, \eqref{coef3} gives $C_2^\top=\det A_1 A_1^{-1}\ne0$,
  and so $b=2$ and $a=-1$.
  
  Next, suppose that $d=2$, so that $\det A_1=0$ and $\beta_2\ne0$.  If  $C_2=0$, then \eqref{coef2}
  would imply that $A_1$ is invertible, a contradiction.  Thus, again we find that $b=2$,
  and so $a=0$.
  
  Finally, assume that $d=1$.  Then $\det A_1=\beta_2=0$ and $\beta_1\ne0$.  If $C_2=C_1=0$, then
  \eqref{coef1} would imply that $A_1$ is invertible, again a contradiction.  Thus,
  $b\in\{1,2\}$, from which follows $a\in\{0,1\}$.
  \end{proof}

\begin{remark}

If
\begin{equation}
A_\infty(t)=
\begin{bmatrix}
t&0&0\\
0&t&1\\
0&-1&0
\end{bmatrix},
\end{equation}
then $b=2$, $d=a=1$.
The other cases can be illustrated with diagonal matrices.
\end{remark}

The next lemma will aid in the formulation of the Cauchy problem at infinity.

\begin{lemma}
\label{cpatinf}
Suppose that $A\in C^2(\rr,\mm^3)$ satisfies $\|A''(t)\|\lesssim t^{-\mu-2}$, $t\gg1$, for some $\mu>0$.
Then
\begin{equation}
\label{solatinf}
 A(t)=A_\infty(t)+\int_t^\infty\int_s^\infty A''(\sigma)d\sigma ds,
\end{equation}
in which
\begin{align}
&A_\infty(t)=    A_0 + t A_1\\
&  A_1=A'(0)+\int_0^\infty A''(\sigma)d\sigma,\\
&  A_0=A(0)-\int_0^\infty\int_s^\infty A''(\sigma)d\sigma ds.
\end{align}
Moreover, the following  estimates hold:
\begin{equation}
\|A^{(k)}(t)-A^{(k)}_\infty(t)\|\lesssim t^{-\mu-k},\quad t\gg1,\quad k=0,1,2.
\end{equation}
\end{lemma}
\begin{proof}
Let $\bar A(t)$ denote the function
on the right-hand side of \eqref{solatinf}.
Given the decay rate for $A''(t)$, the constants $A_0$, $A_1$, and the function $\bar A(t)$  are well-defined.
Note that $\bar A(t)$ satisfies
$\bar A''(t)=A''(t)$, $\bar A(0)= A(0)$, $\bar A'(0)=A'(0)$.
By uniqueness, $A(t)=\bar A(t)$ for all $t\in\rr$.
Therefore, $A(t)-A_\infty(t)=\bar A(t)-A_\infty(t)$ satisfies  the desired estimates.
  \end{proof}

\begin{remark}
This result motivates the condition on $\mu$
in the following theorem.
\end{remark}

\begin{theorem}
\label{compasymst}
Let $A_0,A_1\in\mm^3$,  define $A_\infty(t)=A_0+tA_1$, and assume that \eqref{detgrow} holds.  
Let $a$ and  $d$ be the integers defined in Lemma \ref{asymptoticstates}. 

If $a\ne1$ and $\mu\equiv d(\gamma-1)-a-2 >0$,
  then \eqref{ODE} has a unique global solution  
  \[
  A\in C(\rr,\glt)\cap C^\infty(\rr,\mm^3)
  \]
   with
\begin{equation}
\label{liminfz}
\lim_{t\to\infty}\|A(t)-A_\infty(t)\|=0.
\end{equation}
Moreover, the solution satisfies the decay estimates
\begin{equation}
\|A^{(k)}(t)-A^{(k)}_\infty(t)\|\lesssim t^{-\mu-k},\quad k=0,1,2.
\end{equation}

If $a=d=1$ and $\gamma>5$,
  then \eqref{ODE} has a {unique} global solution  
  \[
  A\in C(\rr,\glt)\cap C^\infty(\rr,\mm^3)
  \]
   with
\begin{equation}
\lim_{t\to\infty}t\|A(t)-A_\infty(t)\|=0.
\end{equation}
Moreover, the solution satisfies the decay estimates
\begin{equation}
\|A^{(k)}(t)-A^{(k)}_\infty(t)\|\lesssim t^{4-\gamma-k},\quad k=0,1,2.
\end{equation}
\end{theorem}

\begin{proof}
Choose $T$ sufficiently large so that $A_\infty(t)$ is invertible, for $t\ge T$.  Then, by the definition \eqref{nonlinform}
and Lemma \ref{asymptoticstates}, we have 
\begin{equation}
\label{mainnonlinest}
\|N(A_\infty(t))\|={(\det A_\infty(t))^{1-\gamma}}{\|A_\infty(t)^{-1}\|}\lesssim t^{a-d(\gamma-1)}=t^{-\mu-2},
\end{equation}
for $t\ge T$.

The function $B\mapsto N(I+B)$ is well-defined and $C^1$ for all $B\in\mm^3$ with  $\|B\|\le1/2$.  Therefore, there exists a constant $C_N>0$
such that
\begin{gather}
\label{nbdd0}
\|N(I+B_1)\|\le C_N
\intertext{and}
\label{nll0}
\|N(I+B_1)-N(I+B_2)\|\le C_N\|B_1-B_2\|,
\end{gather}
for all $\|B_1\|,\|B_2\|\le1/2$.

Now assume that $a\ne1$, so that $\|A_\infty(t)^{-1}\|$ is uniformly bounded.
Define the Banach space
\begin{equation}
\xx(T)=\{B\in C(\rr,\mm^3):\|B\|_T\equiv\sup_{t\ge T}\|B(t)\|<\infty\},
\end{equation}
and the ball
\begin{equation}
\bb_\eps(T)=\{B\in\xx(T):\|B\|_T\le\eps\}.
\end{equation}
Since $A_\infty(\cdot)^{-1}\in\xx(T)$, we may choose $\eps>0$ sufficiently small so that
\begin{equation}
\eps  \|A_\infty(\cdot)^{-1}\|_T \le 1/2.
\end{equation}
Then 
\begin{equation}
\label{Ainvbbd}
\|A_\infty(\cdot)^{-1}B(\cdot)\|_T\le \|A_\infty(\cdot)^{-1}\|_T\|B(\cdot)\|_T\le1/2,
\end{equation}
for all $B\in\bb_\eps(T)$. 

 If $B_i\in\bb_\eps(T)$, $i=1,2$, then
\begin{equation}
\label{nwd}
N(A_\infty(t)+B_i(t))=N(A_\infty(t))N(I+A_\infty(t)^{-1}B_i(t)),
\end{equation}
is well-defined, and by \eqref{mainnonlinest}, \eqref{Ainvbbd}, \eqref{nbdd0}, \eqref{nll0}, the estimates
\begin{multline}
\label{nbdd}
\|N(A_\infty(t)+B_i(t))\|\\
\le\|N(A_\infty(t))\|\|N(I+A_\infty(t)^{-1}B_i(t))\|\lesssim t^{-\mu-2},
\end{multline}
\begin{multline}
\label{nll}
\|N(A_\infty(t)+B_1(t))-N(A_\infty(t)+B_2(t))\|\\
\lesssim t^{-\mu-2}\|B_1(t)-B_2(t)\|,
\end{multline}
hold for all $t\ge T$.

Next, for $B\in\bb_\eps(T)$, define the operator
\begin{equation}
S(B)(t)=\int_t^\infty\int_s^\infty N(A_\infty(\sigma)+B(\sigma))d\sigma ds.
\end{equation}
By \eqref{nwd}, \eqref{nbdd}, and \eqref{nll}, the operator $S$ is well-defined on $\bb_\eps(T)$,
and the following estimates are valid:
\begin{equation}
\label{sbdd}
\|S(B)\|_T\lesssim T^{-\mu},\quad B\in\bb_\eps(T),
\end{equation}
\begin{equation}
\label{sll}
\|S(B_1)-S(B_2)\|_T\lesssim T^{-\mu}\|B_1-B_2\|_T,\quad B_1,B_2\in\bb_\eps(T).
\end{equation}
Therefore, if $T$ is sufficiently large, $S$ is a contraction from $\bb_\eps(T)$ into itself.
By the Contraction Mapping Principle, $S$ has a unique fixed point $B\in\bb_\eps(T)$.  By the definition
of $S$, it follows that this fixed point belongs to $C^\infty([T,\infty),\mm^3)$, and
\begin{equation}
\|B^{(k)}(t)\|\lesssim t^{-\mu-k},\quad t\ge T,\quad k=0,1,2.
\end{equation}
Moreover, $A(t)=A_\infty(t)+B(t)$ solves \eqref{ODE} on the interval $[T,\infty)$, and 
\begin{equation}
\lim_{t\to\infty}\|A(t)-A_\infty(t)\|=\lim_{t\to\infty}\|B(t)\|=0.
\end{equation}

Suppose that $\bar A(t)\in C^2(\rr,\mm^3)$ is a solution of \eqref{ODE}, and let 
\begin{equation}
\bar B(t)=\bar A(t)-A_\infty(t).
\end{equation} 
Assume that $\lim_{t\to\infty}\|\bar B(t)\|=0$.
   Recall that $a\ne1$ and $\|A_\infty(t)^{-1}\|$ is uniformly bounded, so without loss of generality, we may assume
   that the time $T$ defined previously  is  also large  enough so that 
\begin{equation}
\label{innbhd}
\|A_\infty(\cdot )^{-1}\bar B(\cdot )\|_T\le 1/2\quad\text{and}\quad \|\bar B(\cdot )\|_T\le \eps.
\end{equation}
We may  write
\begin{multline}
\bar A''(t)=N(\bar A(t))=N(A_\infty(t)+\bar B(t))
=N(A_\infty(t))N(I+A_\infty(t)^{-1}\bar B(t)),
\end{multline}
and just as in \eqref{nbdd}, we obtain
\begin{equation}
\|\bar A''(t)\|
\lesssim t^{-\mu-2}.
\end{equation}
Therefore, by Lemma \ref{cpatinf}, we find that
\begin{align}
\label{tempie}
\bar A(t)&=\bar A_\infty(t)+\int_t^\infty\int_s^\infty \bar {A}''(\sigma)\;d\sigma ds\\
&=\bar A_\infty(t)+\int_t^\infty\int_s^\infty N(A_\infty(t)+\bar B(t))\;d\sigma ds\\
&=\bar A_\infty(t)+S(\bar B(t)),
\end{align}
where the asymptotic state $\bar A_\infty(t)$ is defined by
\begin{align}
&\bar A_\infty(t)=\bar A_0+t \bar A_1,\\
&\bar A_1=A'(0)+\int_0^\infty N(\bar A(\sigma))d\sigma,\\
&\bar A_0=A(0)-\int_0^\infty\int_s^\infty N(\bar A(\sigma))d\sigma ds.
\end{align}
By Lemma \ref{cpatinf}, we have the estimate
 $\|S(\bar B)(t)\|=\|\bar A(t)-\bar A_\infty(t)\|\lesssim t^{-\mu}$,  and so, using \eqref{liminfz} and \eqref{tempie} we obtain 
\begin{multline}
\limsup_{t\to\infty}\|\bar A_\infty(t)-A_\infty(t)\|\\
\le \limsup_{t\to\infty}\left(\|\bar A_\infty(t)-\bar A(t)\|+\|\bar A(t)-A_\infty(t)\|\right)\\
 \le \limsup_{t\to\infty}\left(\|S(\bar B)(t)\|+\|\bar A(t)-A_\infty(t)\|\right)=0.
\end{multline}
Thus, $\bar A_\infty(t)=A_\infty(t)$.
With this, \eqref{tempie} implies that $\bar B=S(\bar B)$.
By \eqref{innbhd}, $\bar B\in \bb_\eps(T)\subset\xx(T)$, and so by uniqueness of fixed points, $\bar B=B$.
Therefore, $\bar A(t)=A(t)$, for $t\ge T$.  By uniqueness of solutions for \eqref{ODE}, we conclude that
$\bar A(t)=A(t)$, for all $t\in\rr$.

To prove the result in the remaining case, $a=d=1$, we repeat the preceding argument using instead the Banach space
\begin{equation}
\bar\xx(T)=\{B\in C(\rr,\mm^3):\|B\|_T\equiv\sup_{t\ge T}t\|B(t)\|<\infty\}.
\end{equation}
Since $a=1$, we have $\|A_\infty(t)^{-1}\|\sim t$, and the additional decay for the perturbation $B(t)$ provided by the
weight in the norm on $\bar\xx(T)$ ensures that $\|A_\infty(\cdot)^{-1}B(\cdot)\|_T\le1/2$, if $\|B(\cdot)\|_T$ is small.
In order that the fixed point of $S$ lies in $\bar\xx(T)$, we must now have
\begin{equation}
a-d(\gamma-1)<-3,
\end{equation}
 which leads to $\gamma>5$.
  \end{proof}

\begin{corollary}
Let $A_0,A_1\in\mm^3$,  define $A_\infty(t)=A_0+tA_1$, and assume that \eqref{detgrow} holds.  
If $\gamma>5$, then \eqref{ODE} has a unique solution 
\[
A\in C(\rr,\glt)\cap C^\infty(\rr,\mm^3)
\]
 such that
\begin{equation}
\lim_{t\to\infty}\|A(t)-A_\infty(t)\|=0.
\end{equation}
\end{corollary}

\begin{proof}
This follows immediately from Theorem \ref{compasymst} since all the cases in Lemma \ref{asymptoticstates}
are included when $\gamma>5$.
  \end{proof}

\begin{theorem}
\label{waveoperator}
Define $\dd=\glt\times\mm^3$.  

If $\gamma>4/3$, then for any $(A_1,A_0)\in\dd$, 
there exists a unique solution
\[
A\in C(\rr,\glt)\cap C^\infty(\rr,\mm^3)
\]
 of \eqref{ODE} such that \eqref{liminfz} holds.  
Define a mapping ${\mathcal W}_+:\dd\to\dd$ by
\[
{\mathcal W}_+(A_1,A_0)=(A(0),A'(0)).
\]  

If $4/3<\gamma<2$, then ${\mathcal W}_+$ is a bijection.
\end{theorem}
\begin{proof}
The mapping ${\mathcal W}_+$ is well-defined for $\gamma>4/3$, by Theorem \ref{compasymst}.

If ${\mathcal W}_+(A_1,A_0)={\mathcal W}_+(\bar A_1,\bar A_0)$, then 
\begin{equation}
\lim_{t\to\infty}[(A_0+tA_1)-(\bar A_0+t\bar A_1)]=0.
\end{equation}
This implies that $(A_1,A_0)=(\bar A_1,\bar A_0)$, which proves that ${\mathcal W}_+$ is injective.

The theorem will follow if we can show that ${\mathcal W}_+$ is surjective for $\gamma<2$.
Let $(A(0),A'(0))\in\dd$ be arbitrary initial data, and let $A(t)$ be the corresponding global solution of \eqref{ODE}.
By Theorem \ref{compspread}, we have that $\det A(t)\gtrsim t^p$, with $p=3$, for $1<\gamma\le 5/3$, and $p=2/(\gamma-1)$,
for $\gamma\ge5/3$.  Since $\|A(t)\|\sim X(t)^{1/2}$,  Theorem \ref{compspread} says that $\|A(t)\|\sim t$.  It follows that
$\|\cof A(t)\|\lesssim t^2$.  Therefore, we obtain the estimates
\begin{gather}
\label{ainvub}
\|A(t)^{-1}\|=(\det A(t))^{-1}\|\cof A(t)\|\lesssim t^{2-p},\quad t\gg1\\
\intertext{and}
\label{nonlinub}
\|N(A(t))\|=(\det A(t))^{-\gamma}\|\cof A(t)\|\lesssim t^{2-p\gamma},\quad t\gg1.
\end{gather}
From the definition of the exponent $p$,   if $4/3<\gamma<2$, then   $p>2$ and   $2-p\gamma<-2$.
By \eqref{nonlinub} and Lemma \ref{cpatinf}, there exists a unique asymptotic state
$A_\infty(t)=A_0+tA_1$ such that
\begin{equation}
\|A(t)-A_\infty(t)\|\lesssim t^{4-p\gamma},\quad t\gg1.
\end{equation}
Using this and \eqref{ainvub}, we find  that
\begin{equation}
\|A_\infty(t)A(t)^{-1}-I\|\le \|A(t)-A_\infty(t)\|\;\|A(t)^{-1}\|\lesssim t^{4-p\gamma+2-p},
\end{equation}
and so
\begin{equation}
\lim_{t\to\infty}\|A_\infty(t)A(t)^{-1}-I\|=0.
\end{equation}
By continuity of the determinant, we obtain
\begin{equation}
\label{detrat}
\lim_{t\to\infty}\frac{\det A_\infty(t)}{\det A(t)}=\lim_{t\to\infty}\det A_\infty(t)A(t)^{-1}=1.
\end{equation}
Since $\det A(t)\gtrsim t^p$, $p>2$, we find that $\lim_{t\to\infty}t^{-2}\det A_\infty(t)=+\infty$.
It follows that $\det A_\infty(t)$ is cubic, and so, $\det A_1>0$.
Thus, $(A_1,A_0)\in\dd$ and ${\mathcal W}_+(A_1,A_0)=(A(0),A'(0))$.
  \end{proof}

\begin{remark}
Since the system \eqref{ODE} is time-reversible, we can define analogously a bijective operator $\mathcal W_-$ taking
asymptotic states at $-\infty$ to initial data.  Thus, we can construct a (bijective) scattering operator $\Sigma=\mathcal W_+\mathcal W_-^{-1}:\dd\to\dd$,
and we have asymptotic completeness.
\end{remark}

\begin{remark}
It follows from \eqref{detrat} that $\det A(t)\sim \det A_\infty(t)\sim t^3$, for $t\gg1$.
\end{remark}

\section{Asymptotic Behavior of Affine Impressible Swirling and Shear Flow}
\label{swirlsec}

\subsection{Incompressible Swirling Flow}
  We shall now impose the further symmetry 
of   swirling flow upon the system \eqref{IODE}.
The resulting dynamics are governed by equations \eqref{alphaeqn}, \eqref{betaeqn} for two scalar functions $\alpha$ and $\beta$ which measure the
strain and the rotation of the flow, respectively.

\begin{theorem}
Define
\begin{equation}
I_0=
\begin{bmatrix}
1&0\\0&1
\end{bmatrix},
\quad
W_0=
\begin{bmatrix}
0&1\\
-1&0
\end{bmatrix}.
\end{equation}
Let $(\alpha_0',\beta_0')\in\rr^2$ be nonzero.
The global solution 
$
A\in C(\rr,\slt)\cap C^\infty(\rr,\mm^3)
$
 of \eqref{IODE} with initial data
\begin{equation}
A(0)=I,\quad A'(0)=\alpha_0'
\begin{bmatrix}
I_0&\\&-2
\end{bmatrix}
+\beta_0'
\begin{bmatrix}
W_0&\\ 
&0
\end{bmatrix}
\end{equation}
has the block diagonal form
\begin{equation}
\label{block}
A(t)=
\begin{bmatrix}
\alpha(t)I_0&\\
&\alpha(t)^{-2}
\end{bmatrix}
\begin{bmatrix}
\exp(\beta(t)W_0)&\\
&1
\end{bmatrix},
\end{equation}
in which $\alpha(t),\beta(t)\in C^\infty(\rr)$, $\alpha(t)>0$, solve the system
  
\begin{align}
\label{alphaeqn}
&(1+2\alpha(t)^{-6})\left(\frac{\alpha'(t)}{\alpha(t)}\right)'\\
&\qquad+(1-4\alpha(t)^{-6})\left(\frac{\alpha'(t)}{\alpha(t)}\right)^2-(\beta'(t))^2=0,\\
\label{betaeqn}
&\beta''(t)+2\left(\frac{\alpha'(t)}{\alpha(t)}\right)\beta'(t)=0,
\end{align}
with the initial conditions
\begin{equation}
\label{abic}
\begin{aligned}
&\alpha(0)=1,&&\alpha'(0)=\alpha_0'\\
&\beta(0)=0,&&\beta'(0)=\beta_0'.
\end{aligned}
\end{equation}
  
The conserved energy is
\begin{multline}
\label{odeenergy}
e_0\equiv\frac12\tr A'(t)^\top A'(t)\\=(\alpha(t)^2+2\alpha(t)^{-4})\left(\frac{\alpha'(t)}{\alpha(t)}\right)^2+(\alpha(t) \beta'(t))^2\\
=3(\alpha_0')^2+(\beta_0')^2,
\end{multline}
and the curvature is
\begin{equation}
\label{cform}
\kappa(t)=\frac{3\left(\alpha'(t)/\alpha(t)\right)^2-(\beta'(t))^2}{e_0(2\alpha(t)^{-2}+\alpha(t)^{4})^{1/2}}.
\end{equation}

\end{theorem}

\begin{proof}
By uniqueness of solutions to \eqref{IODE}, it is enough to verify that \eqref{alphaeqn}, \eqref{betaeqn}, \eqref{abic} imply that
\eqref{block} defines a solution of \eqref{IODE}.
This will be done by direct computation.  

Note first that $\det A(t)=1$, by \eqref{block}.
Since
\begin{equation}
\label{apform}
A'(t)=\begin{bmatrix}
\alpha'(t) I_0+\alpha(t)\beta'(t)W_0&\\&-2\alpha(t)^{-3}\alpha'(t)
\end{bmatrix}
\begin{bmatrix}
\exp(\beta(t)W_0)&\\&1
\end{bmatrix},
\end{equation}
we have
\begin{equation}
\label{velgradode2.1}
L(t)=A'(t)A(t)^{-1}=\frac{\alpha'(t)}{\alpha(t)}
\begin{bmatrix}
I_0&\\&-2
\end{bmatrix}
+\beta'(t)
\begin{bmatrix}
W_0&\\&0
\end{bmatrix},
\end{equation}
and so $\tr L(t)=0$.  In particular, that  the initial data $(A(0),A'(0))$ satisfies the  necessary compatibility condition
$\tr A'(0)A(0)^{-1}=\tr L(0)=0$.

As in \eqref{ellode}, we have
\begin{equation}
\label{ellode2}
L'(t)=A''(t)A(t)^{-1}-L(t)^2.
\end{equation}
If $A''(t)=\Lambda(t)A(t)^{-\top}$, for some scalar function $\Lambda(t)$, then since $\tr L'(t)=0$, we would have
\begin{equation}
\Lambda(t)=\frac{\tr L(t)^2}{\tr (A(t)^{-\top}A(t)^{-1})}
\end{equation}
so that  \eqref{IODE} holds.
Thus, it is enough to show there exists a scalar function $\Lambda(t)$ such that
\begin{equation}
\label{todo}
A''(t)A(t)^\top=\Lambda(t)I.
\end{equation}

Using \eqref{ellode2}, \eqref{velgradode2.1} and then \eqref{betaeqn}, we have
\begin{align}
A''(t)A(t)^\top=&(L'(t)+L(t)^2)A(t)A(t)^\top\\=&\left\{\left(\frac{\alpha'(t)}{\alpha(t)}\right)'
\begin{bmatrix}
I_0&\\&-2
\end{bmatrix}
+\beta''(t)
\begin{bmatrix}
W_0&\\&0
\end{bmatrix}
+\left(\frac{\alpha'(t)}{\alpha(t)}\right)^2
\begin{bmatrix}
I_0&\\&4
\end{bmatrix}
\right.\\
&\left.
+2\left(\frac{\alpha'(t)}{\alpha(t)}\right)\beta'(t)
\begin{bmatrix}
W_0&\\&0
\end{bmatrix}+(\beta'(t))^2
\begin{bmatrix}
-I_0&\\&0
\end{bmatrix}
\right\}
\begin{bmatrix}
\alpha^2 I_0&\\&\alpha(t)^{-4}
\end{bmatrix}\\
=&\alpha(t)^2\left[\left(\frac{\alpha'(t)}{\alpha(t)}\right)'+\left(\frac{\alpha'(t)}{\alpha(t)}\right)^2-(\beta'(t))^2\right]
\begin{bmatrix}
I_0&\\
&0
\end{bmatrix}\\
&+\alpha(t)^{-4}\left[-2\left(\frac{\alpha'(t)}{\alpha(t)}\right)'+4\left(\frac{\alpha'(t)}{\alpha(t)}\right)^2\right]
\begin{bmatrix}
0&\\&1
\end{bmatrix}.
\end{align}
By \eqref{alphaeqn}, we find that \eqref{todo} holds with
\begin{equation}
\Lambda(t) =\frac{6(\alpha'(t)/\alpha(t))^2-2\beta'(t)^2}{2\alpha(t)^{-2}+\alpha(t)^4}.
\end{equation}

The formula \eqref{cform} for the curvature follows from the one for $\Lambda(t)$ using \eqref{curvaturedef}
and the fact that
\begin{equation}
\tr (A(t)^{-\top}A(t)^{-1}) = 2\alpha(t)^{-2}+\alpha(t)^4.
\end{equation}
The expression \eqref{odeenergy} for the energy follows from \eqref{apform}.
  \end{proof}

The next result provides the precise asymptotic behavior of these  solutions.
Thanks to \eqref{betaeqn}, it is possible to eliminate $\beta$ altogether.  This leaves us with a second order
ODE for $\alpha$ with a parameter $\beta'_0$.  Using the Hamiltonian structure, phase plane analysis allows for
a simple visualization.

\begin{theorem}
\label{incswirl}
Suppose that $\alpha,\beta\in C^\infty(\rr)$, $\alpha>0$, solve the initial value problem \eqref{alphaeqn}, \eqref{betaeqn}, \eqref{abic},
with $e_0=3(\alpha_0')^2+(\beta_0')^2\ne0$.  

If $\beta_0'\ne0$, then 

\begin{align}
 \label{betaform}
 &\beta'(t)=\beta_0'\alpha(t)^{-2},\\
\label{adpgz}
&\alpha''(t)>0,\\
\label{alphapasymp}
&0<e_0^{1/2}-\alpha'(t)\lesssim t^{-2},\quad t\gg1\\
 \label{alphaasymp}
&0<\alpha(t)-\bar\alpha(t)\lesssim t^{-1},\quad t\gg1,\\
&\phantom{0<\alpha(t)}\text{with}\quad
 \bar\alpha(t)=e_0^{1/2}t+1-\int_0^{\infty}(e_0^{1/2}-\alpha'(s))ds,\\
 \label{betaasymp}
&0<|\beta(t)-\bar\beta|\lesssim t^{-1},\quad t\gg1, \\
&\phantom{0<\alpha(t)}\text{with}\quad\bar\beta=\beta_0'\int_0^\infty \alpha(s)^{-2}ds.
\end{align}

There is a nonempty bounded  time interval $(t_1,t_2)\subset\rr$ such that
\begin{equation}
\label{curvatureasymp}
\begin{cases}
\kappa(t)<0,& t\in (t_1,t_2),\\
0<\kappa(t)\lesssim (1+|t|)^{-4},&t\in \rr\setminus [t_1,t_2].
\end{cases}
\end{equation}

The corresponding  solution $A(t)$ of \eqref{IODE} in the form \eqref{block} satisfies 
\begin{equation}
\label{mainincasymp}
A(t)= \begin{bmatrix}
[\bar\alpha(t)I_0-(\beta_0'/e_0^{1/2})W_0]\exp(\bar\beta W_0)&\\
&0
\end{bmatrix}
+\oo(t^{-2}),\quad t\gg1.
\end{equation}
\end{theorem}

\begin{proof}
 Observe that \eqref{betaform} follows immediately from \eqref{betaeqn}.  Thus,
 $(\beta'(t))^2>0$, for $t\in\rr$.

Substitute \eqref{betaform} in \eqref{alphaeqn}    to see that

\begin{equation}
\label{alphadpeqn}
\alpha''=
(1+2\alpha^{-6})^{-1}[6\alpha^{-7}(\alpha')^2+(\beta'_0)^2\alpha^{-3}].
\end{equation}
The right-hand side is strictly positive since $\alpha>0$ and $(\beta_0')^2>0$,
confirming \eqref{adpgz}.

Using \eqref{betaform} in \eqref{odeenergy}, we obtain the conserved energy
\begin{equation}
\label{odeenergyalpha}
(1+2\alpha^{-6})({\alpha'})^2+( \beta'_0)^2\alpha^{-2}=e_0>0,
\end{equation}

from which it follows that  $\min \alpha = |\beta'_0|/e_0^{1/2}$.
Since $\alpha''>0$, there can be no equilibrium solutions.
Therefore, the trajectory
$(\alpha(t),\alpha'(t))$, $t\in\rr$, traces the entire energy level curve \eqref{odeenergyalpha} in the direction of increasing $\alpha'$,
as  depicted  in the phase diagram in Figure \ref{phasediagram1}.

Using \eqref{betaform} to eliminate $\beta$ in \eqref{cform}, we see that the sign of the curvature
is determined by the sign of the quantity $3(\alpha'\alpha)^2-(\beta_0')^2$.  Since $\beta_0'\ne0$,  a bounded,
 connected portion of the solution trajectory necessarily lies in a region where $\kappa<0$, as
illustrated in Figure \ref{phasediagram1}.  This establishes the existence of the time interval claimed in \eqref{curvatureasymp}.

\begin{figure}
\label{phasediagram1}
\caption{Phase diagram in the case $\beta_0'=1/2$ showing   the energy level set $e_0=1$ and
 the regions of positive and negative curvature for $A(t)$.}
 \ \\
 \ \\
\setlength\unitlength{1mm}
\begin{center}
\begin{overpic}[scale=.5]
{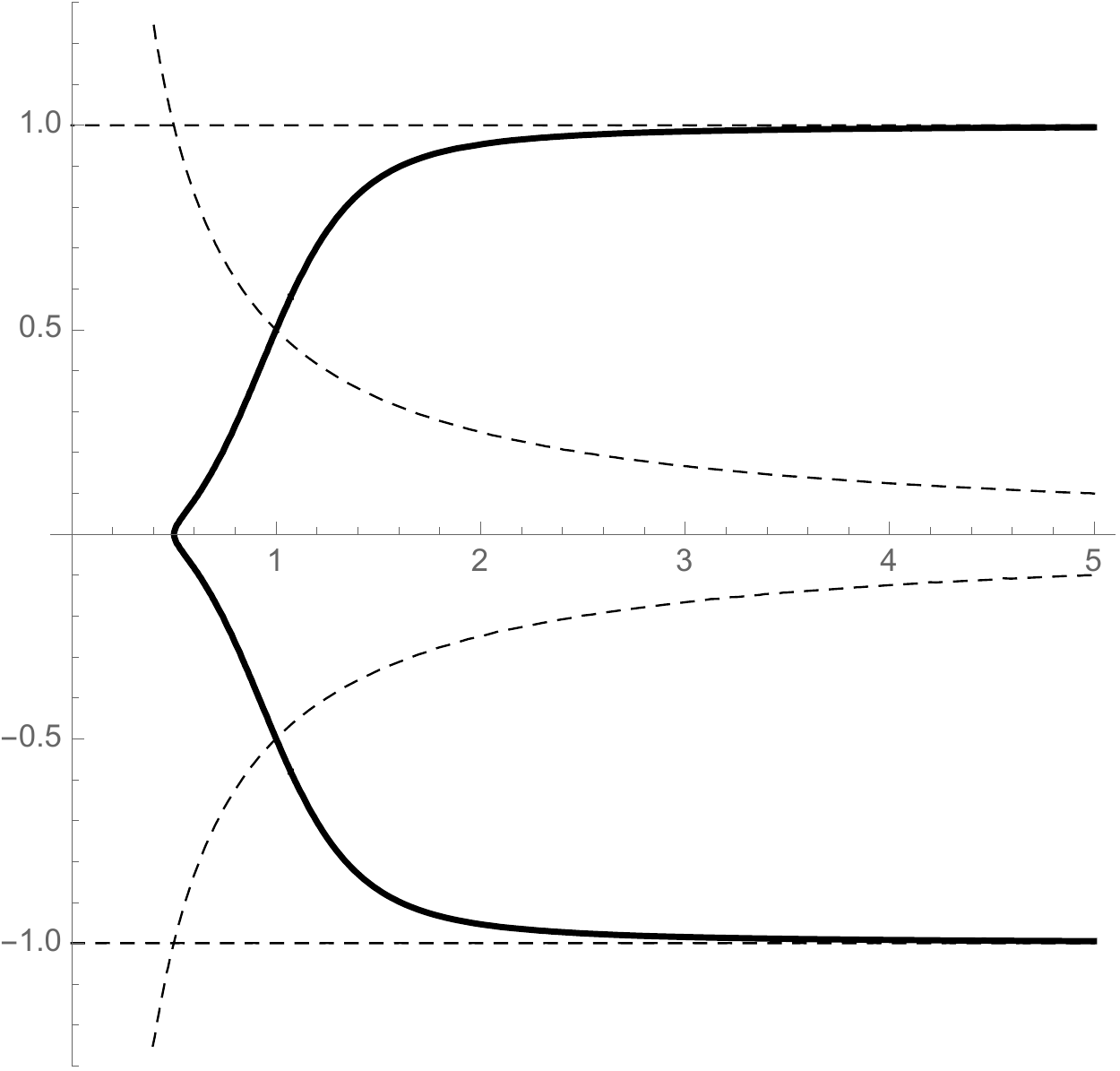}
\put(65,69){$\kappa>0$}
\put(65,26){$\kappa>0$}
\put(27,51){$\kappa<0$}
\put(102,47){$\alpha$}
\put(5.2,98){$\alpha'$}
\thicklines
\put(75,11.8){\vector(-1,0){3}}
\put(72,84.2){\vector(1,0){3}}
\end{overpic}
\end{center}
\end{figure}
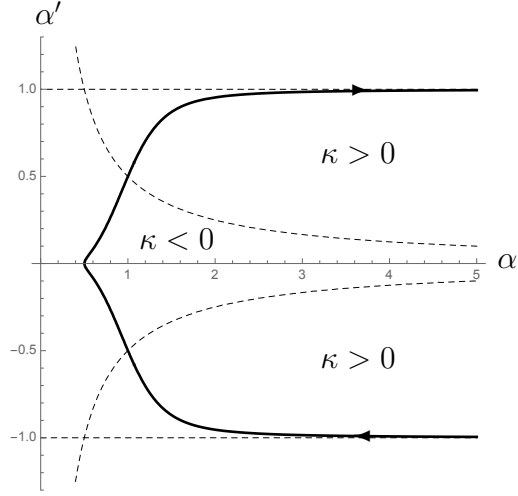

Next, we consider the asymptotic behavior  as $t\to+\infty$.
The  behavior as $t\to-\infty$ is similar and can be obtained simply by reversing time.
More precisely, let $t_0\in\rr$ be the unique time such that $\alpha'(t_0)=0$.
Since the equation \eqref{alphadpeqn} is time reversible, we have that $\alpha(t-t_0)=\alpha(t_0-t)$
for all $t\in\rr$.

Since $(\alpha(t),\alpha'(t))\to(\infty,e^{1/2})$, as $t\to\infty$, we have that $\alpha'(t)> e_0^{1/2}/2$ for $t\gg1$,
and thus 

\begin{equation}
\label{alphalb}
\alpha(t)> e_0^{1/2}t/2,\quad t\gg1. 
\end{equation}
From \eqref{odeenergyalpha}, we obtain
\begin{equation}
\label{odeenergyalpha1}
\alpha'(t)=\psi(\alpha(t)),\quad \psi(\alpha)\equiv \left(\frac{e_0-(\beta'_0)^2\alpha^{-2}}{1+2\alpha^{-6}}\right)^{1/2},\quad t\gg1.
\end{equation}
Perform an expansion of $\psi(\alpha)$  for $\alpha\gg1$:
\begin{equation}
\label{psiexpan}
\psi(\alpha)=e_0^{1/2}-\frac{(\beta_0')^2}{2e_0^{1/2}\alpha^2}+\oo(\alpha^{-4}),\quad \alpha\gg1.
\end{equation}
By \eqref{alphalb}, \eqref{odeenergyalpha1}, \eqref{psiexpan},  we have
\begin{equation}
\label{asymp1}
0<e_0^{1/2}-\alpha'(t)=e_0^{1/2}-\psi(\alpha(t))\lesssim t^{-2},\quad t\gg1.
\end{equation}
which proves \eqref{alphapasymp}.
Integration   of \eqref{asymp1} over an interval of the form $[t,\infty)$, $t\gg1$, shows that 
\begin{equation}
0<\int_t^\infty(e_0^{1/2}-\alpha'(s))ds\lesssim t^{-1}.
\end{equation}
The estimate \eqref{alphaasymp} follows from this since 
\begin{align}
\label{alphadecomp}
\int_t^\infty(e_0^{1/2}&-\alpha'(s))ds
\\&=\int_0^\infty(e_0^{1/2}-\alpha'(s))ds-\int_0^t(e_0^{1/2}-\alpha'(s))ds\\
\\&=\alpha(t)-\bar\alpha(t).
\end{align}

We obtain \eqref{betaasymp} in the same way.
By \eqref{alphalb}, we have
\begin{equation}
0<\int_t^\infty\alpha(s)^{-2}ds\lesssim t^{-1},\quad t\gg1,
\end{equation}
and by \eqref{betaform},
\begin{equation}
\label{detadecomp}
\beta_0'\int_t^\infty\alpha(s)^{-2}ds=\bar\beta
-\int_0^t \beta'(s)ds=\bar\beta-\beta(t).
\end{equation}

The upper bound in \eqref{curvatureasymp} for the curvature follows from \eqref{cform}, the bounds \eqref{alphapasymp}, \eqref{alphaasymp}, \eqref{betaasymp},
and their analogs for $t\ll-1$.

It remains to verify the estimate \eqref{mainincasymp}, which is equivalent to showing that

\begin{equation}
\label{maininc1}
\alpha(t)\exp(\beta(t)W_0)=[\bar\alpha(t)I-(\beta_0'/e_0^{1/2})W_0]\exp(\bar\beta W_0)
+\oo( t^{-2}),
\end{equation}
and
\begin{equation}
\label{maininc2}
\alpha(t)^{-2}=\oo( t^{-2}),
\end{equation}

for $ t\gg1$.  Of course, \eqref{maininc2} is a consequence of \eqref{alphalb}.  To prove \eqref{maininc1},
we will need to refine our asymptotic formulas slightly.

To this end, let us write
\begin{equation}
\alpha(t)=\bar\alpha(t)+\dot\alpha(t)\quad \text{and}
\quad \beta(t)=\bar\beta +\dot\beta(t).
\end{equation}
(The notation $\dot\alpha$, $\dot{\beta}$  denotes a perturbation and not a derivative here.)
By \eqref{alphadecomp}, \eqref{asymp1}, \eqref{psiexpan}, and \eqref{alphalb}, we have that

\begin{gather}
\label{alphadot}
\dot\alpha(t)=(\beta_0')^2/(2e_0^{1/2})\int_t^\infty \alpha^{-2}(s)ds+\oo(t^{-3}),\quad t\gg1,\\
\intertext{and by \eqref{detadecomp}}
\label{betadot}
\dot\beta(t)=- \beta_0'\int_t^\infty\alpha^{-2}(s)ds,\quad t\gg1.
\end{gather}

Now by \eqref{alphaasymp},  we have that $\dot\alpha(t)\lesssim t^{-1}$.  Therefore,
for $t\gg1$, it follows that
\begin{equation}
\label{alpham2}
\alpha(t)^{-2}=\bar\alpha(t)^{-2}(1+\dot\alpha(t)/\bar\alpha(t))^{-2}
 =\bar\alpha(t)^{-2}+\oo(t^{-4}).
\end{equation}
Cycling this estimate into \eqref{alphadot} and \eqref{betadot}, we conclude that
\begin{gather}
\label{alphadot2}
\dot\alpha(t)=(\beta_0')^2/(2e_0\bar\alpha(t))+\oo(t^{-3}),\quad t\gg1,\\
\intertext{and}
\label{betadot2}
\dot\beta(t)=-\beta_0'/(e_0^{1/2}\bar\alpha(t))+\oo(t^{-3}),\quad t\gg1.
\end{gather}

These estimates 
imply that
  
\begin{align}
\label{alphaexpansion}
&\alpha(t)=\bar\alpha(t)+(\beta_0')^2/(2e_0\bar\alpha(t))+\oo(t^{-3}),\quad t\gg1\\
\label{betaexpansion}
&\beta(t)=\bar\beta-\beta_0'/(e_0^{1/2}\bar\alpha(t))+\oo(t^{-3}),\quad t\gg1.
\end{align}
By \eqref{betaasymp}, we have $|\dot\beta(t)|\lesssim t^{-1}$, and so using \eqref{betadot2}, we obtain
\begin{align}
\exp(\beta(t)W_0)&=\exp(\bar\beta W_0)\exp(\dot\beta(t)W_0)\\
&=\exp(\bar\beta W_0)(I_0+\dot\beta(t)W_0+\textstyle{\frac12}\dot\beta(t)^2W_0^2+\oo(t^{-3}))\\
\label{expbetaexpansion}
&=\exp(\bar\beta W_0)(I_0+\dot\beta(t)W_0-\textstyle{\frac12}\dot\beta(t)^2I_0)+\oo(t^{-3})\\
&=\exp(\bar\beta W_0) (I_0  -\beta_0'/(e_0^{1/2}\bar\alpha(t))W_0-\textstyle{\frac12} (\beta_0')^2/(e_0\bar\alpha(t)^2)I_0)\\
&\phantom{=\exp(\bar\beta W_0) (I_0  -\beta_0'/(e_0^{1/2}\bar\alpha(t))W_0-\textstyle{\frac12} (\beta_0')^2/(e_0\alpha}  
+\oo(t^{-3}).
\end{align}
  
We are now ready to prove \eqref{maininc1}.  Since the error is $\oo(t^{-2})$,
we only need to keep track of the terms
of order $\bar\alpha(t)^k$, for $k=1,0,-1$.  Multiplication of the
expansions \eqref{alphaexpansion} and \eqref{expbetaexpansion} yields \eqref{maininc1},
because the terms with $\bar\alpha(t)^{-1}$ drop out.
This completes the proof of \eqref{mainincasymp}.
  \end{proof}

\begin{remark}
The existence of an asymptotic state $A_\infty(t)$ and the decay rate $t^{-2}$ in \eqref{mainincasymp}
could also be deduced by applying Lemma \ref{cpatinf} to \eqref{IODE} and using \eqref{curvatureasymp}.
However, the argument of Theorem \ref{incswirl} also provides the  form  of $A_\infty(t)$.
These asymptotic estimates can be continued to include further terms.
\end{remark}  

\begin{remark}

 The rescaled fluid domain collapses to a circular pancake as $t\to\infty$, that is,
$\barominf=\{(x_1,x_2,0)\in\rr^3: x_1^2+x_2^2\le e_0\}$.
\end{remark}

\begin{remark}
The vorticity  is given by  the operator
\begin{equation}
\frac12\omega(t,x)\times=\frac12\omega(t)\times=\frac12(L(t)-L(t)^\top)=\beta'(t)
\begin{bmatrix}
W_0&\\&0
\end{bmatrix}=\oo(t^{-2}).
\end{equation}
\end{remark}

\subsection{Irrotational Incompressible Swirling Flow}

We now consider the case where the  parameter $\beta'_0$ vanishes.  This  eliminates the vorticity and significantly changes the
character of the phase diagram for $(\alpha,\alpha')$.

\begin{theorem}
\label{incirrot}
Suppose that $\alpha,\beta\in C^\infty(\rr)$, $\alpha>0$, solve the initial value problem \eqref{alphaeqn}, \eqref{betaeqn}, \eqref{abic}.

If $\alpha_0'<0$ and $\beta_0'=0$, then 

\begin{align}
\label{betavan}
&\beta(t)\equiv0,\\
\label{adpgz2}
&\alpha''(t)>0,\\
\label{alphapdnt}
&0<\alpha'(t)+e_0^{1/2}\lesssim (1+|t|)^{-6},\quad t<0, \\
 \label{alphantasymp}
&0<\alpha(t)-e_0^{1/2}|t|-\alpha_{-\infty}\lesssim (1+|t|)^{-5},\quad t<0,\\
&\phantom{0<\alpha(t)}\text{with}\quad
 \alpha_{-\infty}=1-\int_{-\infty}^0(\alpha'(s)+e_0^{1/2})ds,\\
 \label{alphamtprimeasymp}
&0<(\alpha(t)^{-2})'-(2e_0)^{1/2}\lesssim (1+t)^{-3},\quad t>0,\\
 \label{alphamtasymp}
&0<\alpha(t)^{-2}-(2e_0)^{1/2}t-\alpha_{\infty}\lesssim (1+t)^{-2},\quad t>0,\\
&\phantom{0<\alpha(t)}\text{with}\quad
 \alpha_{\infty}=1-\int_0^{\infty}[(\alpha(s)^{-2})'-(2e_0)^{1/2}]ds.
 \end{align}

Finally, the curvature is everywhere positive and
\begin{equation}
\label{}
\begin{cases}
0<\kappa(t)\lesssim (1+|t|)^{-4},& t<0,\\
0<\kappa(t)\lesssim (1+t)^{-5/2},&t>0.
\end{cases}
\end{equation}

By time reversal, corresponding statements hold when $\alpha_0'>0$.
\end{theorem}

\begin{proof}
The vanishing of $\beta(t)$, \eqref{betavan}, follows from \eqref{betaeqn}, since $\beta(0)=\beta'(0)=0$.
The positivity of $\alpha''(t)$, \eqref{adpgz2}, follows from \eqref{alphaeqn}, as shown in \eqref{alphadpeqn},
and thus $\alpha'(t)$ is again strictly increasing.

The vanishing of $\beta$ in \eqref{odeenergy} implies that 
\begin{equation}
\label{odeenergybpz}
\alpha'(t)=\pm\left(\frac{e_0}{1+2\alpha(t)^{-6}}\right)^{1/2}.
\end{equation}
Thus, the energy level curve now has two distinct components, as illustrated in
Figure \ref{phasediagram2}.  By assumption $\alpha_0'<0$, so we have $(\alpha(t),\alpha'(t))\to(\infty,-e_0^{1/2})$, 
as $t\to -\infty$, as before,
but now $(\alpha(t),\alpha'(t))\to(0,0)$, as $t\to\infty$.

\begin{figure}
\label{phasediagram2}
\caption{Phase diagram in the case  $\beta_0'=0$ showing the two components of the energy level set $e_0=1$.}
\ \\
\ \\
\setlength\unitlength{1mm}
\begin{center}
\begin{overpic}[scale=.5]
{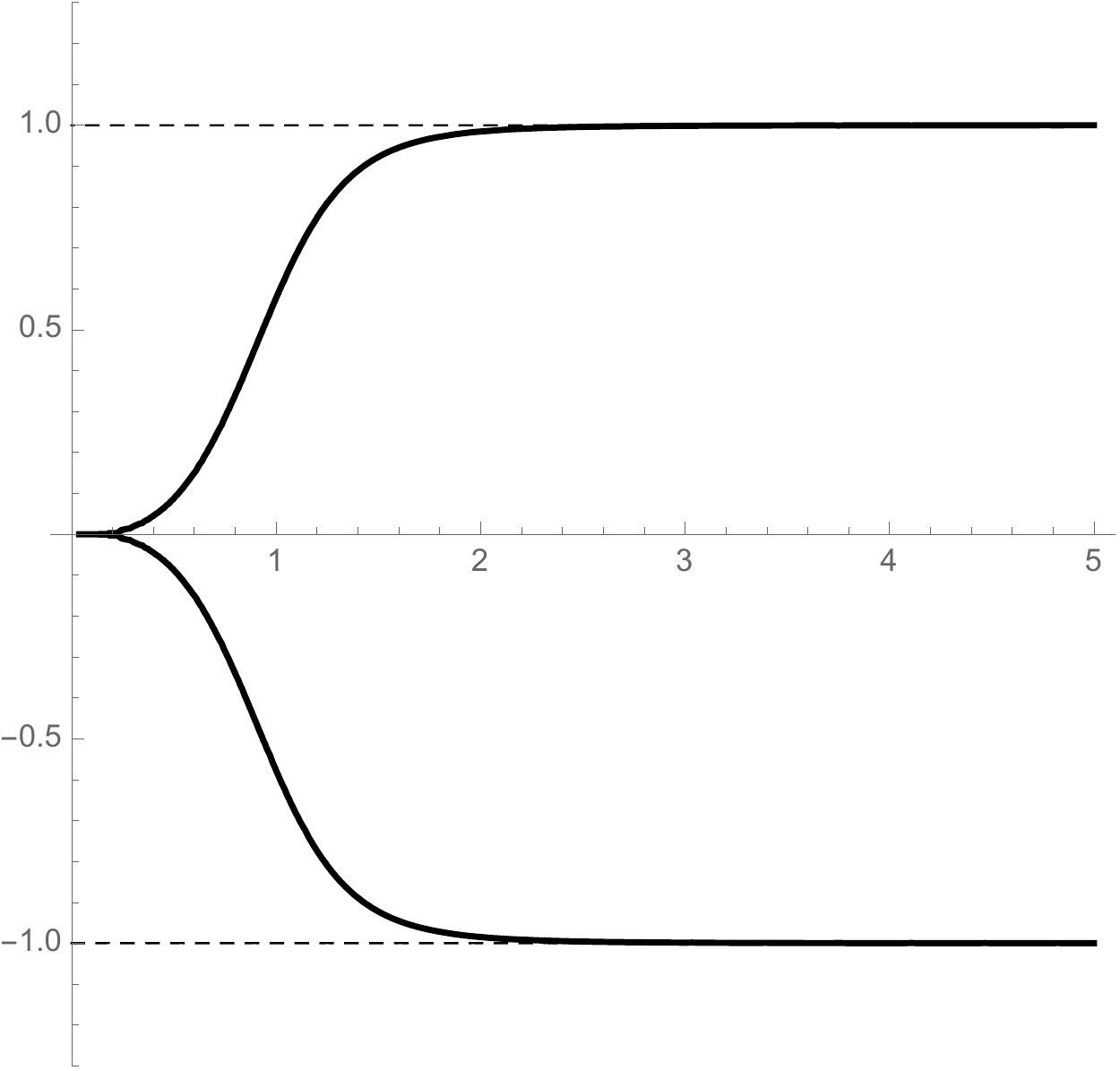}
\put(102,47){$\alpha$}
\put(5.2,98){$\alpha'$}
\put(6.2,48){\circle{2}}
\thicklines
\put(60,11.5){\vector(-1,0){3}}
\put(57,84.4){\vector(1,0){3}}
\end{overpic}
\end{center}
\end{figure}
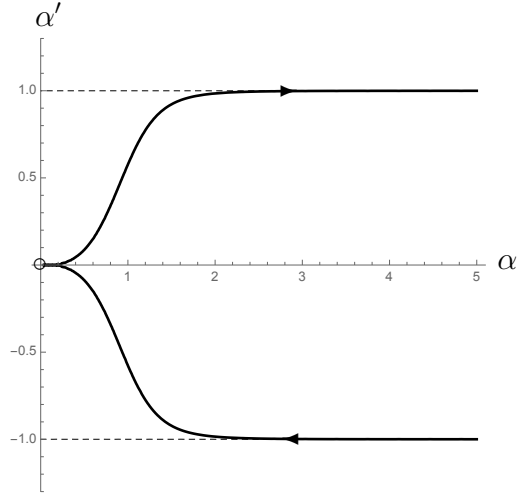

Since $\alpha'(t)\downarrow-e_0^{1/2}$, as $t\to-\infty$, we have that 
\begin{equation}
\label{alphalbnt}
\alpha(t)\gtrsim (1+|t|),\quad\text{for }\quad t<0.
\end{equation}
Since $\alpha'_0<0$, we are on the lower branch of the energy curve \eqref{odeenergybpz},
 so
\begin{equation}
\alpha'(t)+e_0^{1/2}=-\left(\frac{e_0}{1+2\alpha(t)^{-6}}\right)^{1/2}+e_0^{1/2}=\frac{2e_0^{1/2}}{2+\alpha(t)^6}.
\end{equation}
Combining this with \eqref{alphalbnt}, we have shown \eqref{alphapdnt}.
The bound \eqref{alphantasymp} follows from \eqref{alphapdnt} by integration.

Next, we focus on the behavior for large positive times.  From \eqref{odeenergybpz}, we have
\begin{equation}
\label{alphamtp}
(\alpha(t)^{-2})'=-2\alpha(t)^{-3}\alpha'(t)=\left(\frac{2e_0}{1+\alpha(t)^6/2}\right)^{1/2}\gtrsim1,\quad t>0,
\end{equation}
since $\alpha(t)\downarrow0$, as $t\to\infty$.  From \eqref{alphamtp}  we obtain the bound
\begin{equation}
\label{alphadecay}
\alpha(t)\lesssim (1+t)^{-1/2},\quad t>0.
\end{equation}
Again by \eqref{alphamtp}, we have
\begin{equation}
0<(2e_0)^{1/2}-(\alpha(t)^{-2})'=\alpha(t)^6\psi(\alpha(t)),
\end{equation}
where $\psi(\alpha(t))\to (e_0)^{1/2}/4>0$, as $\to\infty$.
Together with \eqref{alphadecay}, this proves the estimate \eqref{alphamtprimeasymp}.
As above, \eqref{alphamtasymp} follows directly  from \eqref{alphamtprimeasymp}.

The statements about the curvature follow from the formula \eqref{cform}, with $\beta=0$,
and the estimates \eqref{alphalbnt}, \eqref{alphadecay}.
  \end{proof}

\begin{remark}
For irrotational swirling flow with $\alpha_0'<0$, the asymptotic fluid domains are
\begin{equation}
\overline{\Omega}_\infty=\{(0,0,x_3)\in\rr^3:x_3^2< 2e_0\}
\end{equation}
and
\begin{equation}
\overline{\Omega}_{-\infty}=\{(x_1,x_2,0)\in\rr^3:x_1^2+x_2^2< e_0\}.
\end{equation}
\end{remark}
\subsection{Shear Flow}

Suppose that $M\in\mm^3$ is nilpotent, so that $M^3=0$.  Define $A(t)=I+tM$.  
Since  the eigenvalues of $M$ vanish, the eigenvalues of $A(t)$ are equal to unity, for all $t\in\rr$.
It follows that $\det A(t)=1$ and so, $A\in C(\rr,\slt)\cap C^\infty(\rr,\mm^3)$.
Since $A(t)$ is a line in $\mm^3$, its curvature vanishes.  This can also be verified directly from
the formula \eqref{curvaturedef} for $\kappa(t)$.  Since $A''(t)=0$ and $\kappa(t)=0$,
for all $t\in\rr$, we see that $A(t)$ is a solution of \eqref{IODE} whose corresponding pressure vanishes
identically.
More generally, we could take $A(t)=(I+tM)A_0$, for an arbitrary element $A_0\in\slt$.

Consider, for example,  $M=e_2\otimes e_1$, whence $A(t)$ gives rise to a classical shear flow.
In this case, the rescaled asymptotic fluid domain is a line segment \[\barominf=\{x=(x_1,0,0):|x_1|<1\}.\]

Or, if $M=e_2\otimes e_1 + e_3\otimes e_2$, then the rescaled limit is a disk
\[\barominf=\{x=(x_1,x_2,0):x_1^2+x_2^2<1\}.\]

\bibliography{arXiv-02-26-17}
\bibliographystyle{plain}

\end{document}